\documentclass{amsart}%
\usepackage{amssymb}
\usepackage{amsfonts}
\usepackage{amsmath}
\usepackage{graphicx}
\usepackage[dvips]{color}
\usepackage{amsthm}
\usepackage{array}
\usepackage[latin1]{inputenc}%
\newtheorem{theorem}{Theorem}[section]
\theoremstyle{plain}

\newtheorem{definition}[theorem]{Definition}
\newtheorem{example}[theorem]{Example}

\newtheorem{lemma}[theorem]{Lemma}

\newtheorem{proposition}[theorem]{Proposition}
\newtheorem{remark}[theorem]{Remark}

\numberwithin{equation}{section}
\begin{document}
\title[Semi-Quasielliptic Operators Over $p$-adic Fields ]{Pseudo-differential Operators with Semi-Quasielliptic Symbols Over $p$-adic
Fields }
\author{J. Galeano-Peñaloza}
\address{Centro de Investigación y de Estudios Avanzados del Instituto Politécnico Nacional\\
Departamento de Matemáticas- Unidad Querétaro\\
Libramiento Norponiente \#2000, Fracc. Real de Juriquilla. Santiago de
Querétaro, Qro. 76230\\
México\\
and Universidad Nacional de Colombia, Departamento de Matemáticas, Ciudad
Universitaria, Bogotá, Colombia}
\email{jgaleano@math.cinvestav.mx and jgaleanop@unal.edu.co }
\author{W. A. Zúñiga-Galindo}
\address{Centro de Investigación y de Estudios Avanzados del Instituto Politécnico Nacional\\
Departamento de Matemáticas- Unidad Querétaro\\
Libramiento Norponiente \#2000, Fracc. Real de Juriquilla. Santiago de
Querétaro, Qro. 76230\\
México}
\email{wazuniga@math.cinvestav.edu.mx}
\thanks{The second author was partially supported by CONACYT under Grant \# 127794.}

\begin{abstract}
In this article, we study pseudo-differential operators with a
semi-quasielliptic symbols, which are extensions of operators of the form
\[
\boldsymbol{F}(D;\alpha;x)\phi=\ \allowbreak\mathcal{F}^{-1}(|F(\xi
,x)|_{p}^{\alpha}\mathcal{F}\left(  \phi\right)  ),
\]
where $\alpha>0$, $\phi$ is a function of Lizorkin type and $\mathcal{F}$
denotes the Fourier transform, and $F\left(  \xi,x\right)  =f\left(
\xi\right)  +\sum_{\boldsymbol{k}}c_{\boldsymbol{k}}\left(  x\right)
\xi^{\boldsymbol{k}}\in\mathbb{Q}_{p}\left[  \xi_{1},\ldots,\xi_{n}\right]  $,
where $f(\xi)$ is a quasielliptic polynomial of degree $d$, and each
$c_{\boldsymbol{k}}\left(  x\right)  $ is a function from $\mathbb{Q}_{p}^{n}$
into $\mathbb{Q}_{p}$ satisfying $||c_{\boldsymbol{k}}(x)||_{L^{\infty}%
}<\infty$. We determine the function spaces where the equations
$\boldsymbol{F}(D;\alpha;x)u=v$ have solutions. We introduce the space of
infinitely pseudo-differentiable functions with respect to a
semi-quasielliptic operator. By using these spaces we show the existence of a
regularization effect for certain parabolic equations over $p$-adics.

\end{abstract}
\maketitle

\section{Introduction}

The $p$-adic pseudo-differential equations have received a lot of attention
during the last ten years due mainly to fact that some of these equations
appeared in certain new physical models, see e.g. \cite[and references
therein]{D-K-K}, \cite[and references therein]{A-K-S-1}, \cite[and references
therein]{Koch}, \cite{V-V-Z}. There are two motivations behind this article.
The first one is to understand the connection between the $p$-adic Lizorkin
spaces introduced by Albeverio, Khrennikov, Shelkovich \ in \cite{A-K-S} and
the function spaces introduced by Rodríguez-Vega and Zúñiga-Galindo\ in
\cite{R-Z1}. In \cite[and references therein]{G-V}\ Gindikin and Volevich
developed the method of Newton's polyhedron for some problems in the theory of
partial differential equations. Since in the $p$-adic setting the method of
Newton's polyhedron has been widely used in arithmetic and geometric problems,
see. e.g. \cite[and references therein]{VZ}, \cite[and references
therein]{ZG4}, it is natural to ask: does a $p$-adic counterpart of the
Gindikin-Volevich results exist? This is our second motivation.

In this article, motivated by the work of Gindikin and Volevich \cite{G-V}, we
study pseudo-differential operators with a \textit{semi-quasielliptic
symbols}, which are extensions of operators of the form
\[
\boldsymbol{F}(D;\alpha;x)\phi=\mathcal{F}^{-1}(|F(\xi,x)|_{p}^{\alpha
}\mathcal{F}\left(  \phi\right)  ),
\]
where $\alpha>0$, $\phi$ is a function of Lizorkin type and $\mathcal{F}$
denotes the Fourier transform, and
\[
F\left(  \xi,x\right)  =f\left(  \xi\right)  +{\sum\limits_{\boldsymbol{k}}%
}c_{\boldsymbol{k}}\left(  x\right)  \xi^{\boldsymbol{k}}\in\mathbb{Q}%
_{p}\left[  \xi_{1},\ldots,\xi_{n}\right]  ,
\]
where $f(\xi)$ is a \textit{quasielliptic polynomial} of degree $d$, see
Definition \ref{def_qelliptic}, and each $c_{\boldsymbol{k}}\left(  x\right)
$ is a function from $\mathbb{Q}_{p}^{n}$ into $\mathbb{Q}_{p}$ satisfying
$||c_{\boldsymbol{k}}(x)||_{L^{\infty}}<\infty$, see Definition
\ref{def_sqelliptic}. We determine the function spaces where the equations
$\boldsymbol{F}(D;\alpha;x)u=v$ have solutions.

In Sections \ref{Sect1}-\ref{Sect2}, we study the properties of
\textit{quasielliptic and semi-quasielliptic polynomials}. We establish
similar inequalities for the symbols like ones presented in \cite{G-V}, see
Theorems \ref{Theo1}, \ref{Teo3}. The techniques used in \cite{G-V} cannot be
used directly in the $p$-adic case, as far as we understand, and then we
follow a completely different approach.

In Section \ref{Sect3}, we study equations of type $\boldsymbol{f}%
(D;\alpha)u=v$, whose symbol is a quasielliptic polynomial. These
quasielliptic operators are a generalization of the elliptic operators
introduced in \cite{ZG3}, see also \cite{R-Z1}, \cite{H}. We introduce new
function spaces $\widetilde{H}_{\beta}(\mathbb{Q}_{p}^{n})$ which are
constructed as completions of certain Lizorkin spaces with respect \ to a norm
$||\cdot||_{\beta}$, see Definition \ref{def_space}. The spaces $\widetilde
{H}_{\beta}(\mathbb{Q}_{p}^{n})$ are $p$-adic analogs of the function spaces
introduced in \cite{G-V}. We show that if $v\in\widetilde{H}_{\beta-\alpha}$,
with $\beta\geq\alpha$, then $\boldsymbol{f}(D;\alpha)u=v$ has an unique
solution $u\in\widetilde{H}_{\beta}$, see Theorem \ref{surjective}. We also
introduce the space $\widetilde{H}_{\infty}(\mathbb{Q}_{p}^{n})$, which is the
space of \textit{infinitely pseudo-differentiable functions} \ \textit{with
respect to} $\boldsymbol{f}(D;\alpha)$, we show that $\boldsymbol{f}%
(D;\alpha):\widetilde{H}_{\infty}\rightarrow\widetilde{H}_{\infty}$,
$\phi\rightarrow\boldsymbol{f}(D;\alpha)\phi$ gives a bicontinuous isomorphism
between locally convex vector spaces, see Theorem \ref{Theo6}. A similar
result is also valid for the Taibleson operator, see Remark \ref{Taibleson}.
The space $\widetilde{H}_{\infty}$ contains functions of type $\int
_{\mathbb{Q}_{p}^{n}}\Psi\left(  x\cdot\xi\right)  e^{-t\left\vert f\left(
\xi\right)  \right\vert _{p}^{\alpha}}d^{n}\xi$, $\int_{\mathbb{Q}_{p}^{n}%
}\Psi\left(  x\cdot\xi\right)  e^{-t\left\Vert \xi\right\Vert _{p}^{\alpha}%
}d^{n}\xi$, see Sections \ref{Sect_family}, \ref{Sect_reg}, which are
solutions of the $n$-dimensional heat equation, \cite{V-V-Z}, \cite{Koch},
\cite{ZG3}, \cite{R-Z0}, see Theorem \ref{Theo7A}\ and Remark \ref{note5}.

By using the space $\widetilde{H}_{\infty}$, we establish the existence of a
\textit{regularization effect }in certain parabolic equations. More precisely,
if $\boldsymbol{f}\left(  D,\alpha\right)  $ is an elliptic
pseudo-differential operator, or the Taibleson operator, and if $\varphi\in
L^{2}\left(  \mathbb{Q}_{p}^{n}\right)  \cap C_{b}\left(  \mathbb{Q}_{p}%
^{n},\mathbb{C}\right)  $, then the solution of the Cauchy problem:
\[
\left\{
\begin{array}
[c]{l}%
\frac{\partial u\left(  x,t\right)  }{\partial t}+\left(  \boldsymbol{f}%
\left(  D,\alpha\right)  u\right)  \left(  x,t\right)  =0,\text{ }%
x\in\mathbb{Q}_{p}^{n},\text{ }t>0,\\
\\
u\left(  x,0\right)  =\varphi\left(  x\right)  ,
\end{array}
\right.
\]
$u\left(  x,t\right)  $ belongs to $\widetilde{H}_{\infty}$, for every fixed
$t>0$, see Theorem \ref{Theo7} and Remark \ref{Taibleson2}.

In Section \ref{Sect7} we determine the function spaces $\widetilde{H}%
_{(\beta,M_{0})}$, see Definition \ref{def_space2}, where the equation
$\boldsymbol{F}(D;\alpha;x)u=v$ has a unique solution, see Theorem
\ref{Theo9}. Finally, we construct the space of infinitely
pseudo-differentiable functions with respect to $\boldsymbol{F}(D;\alpha;x)$,
see Theorem \ref{Theo10}.

We now return to our initial motivations.\ We introduce four new types of
function spaces: $\widetilde{H}_{\beta}$, $\widetilde{H}_{\infty}$,
$\widetilde{H}_{(\beta,M_{0})}$, $\widetilde{H}_{(\infty,M_{0})}$. Each space
of type $\widetilde{H}_{\beta}$ contains a dense copy of the Lizorkin space of
second class introduced in \cite{A-K-S}, furthermore, the spaces
$\widetilde{H}_{\beta}$ \ are similar, but not equal, to the $l$-singular
Sobolev spaces $\mathcal{H}^{l}$ introduced in \cite{R-Z1}, the other spaces
are new. It is important to mention that the regularization effect in
parabolic equations appear only in the spaces $\widetilde{H}_{\infty}$, and
\ not in the Lizorkin \ spaces. Finally, after the results presented here, we
strongly believe that the geometric point of view of Gindikin and Volevich in
PDE's have a meaningful and non-trivial counterpart in the $p$-adic setting.

\textbf{Acknowledgement.} The authors wish to thank to Sergii Torba for his
careful reading of the paper and for his constructive remarks.

\section{Preliminaries}

\subsection{Fixing the notation}

We denote by $\mathbb{Q}_{p}$ the field of $p-$adic numbers and by
$\mathbb{Z}_{p}$ the ring of $p-$adic integers. For $x\in\mathbb{Q}_{p}$,
$ord\left(  x\right)  \in\mathbb{Z}\cup\left\{  +\infty\right\}  $ denotes the
valuation of $x$, and $\left\vert x\right\vert _{p}=p^{-ord\left(  x\right)
\text{ }}$its absolute value. We extend this absolute value to $\mathbb{Q}%
_{p}^{n}$ by taking $\left\Vert x\right\Vert _{p}=\max_{i}\left\vert
x_{i}\right\vert _{p} $\ for $x=\left(  x_{1},\ldots,x_{n}\right)
\in\mathbb{Q}_{p}^{n}$.

Let $\Psi:\mathbb{Q}_{p}\rightarrow\mathbb{C}^{\times}$ be the additive
character defined by%
\[%
\begin{array}
[c]{lll}%
\Psi:\mathbb{Q}_{p}\rightarrow\mathbb{Q}_{p}/\mathbb{Z}_{p}\hookrightarrow &
\mathbb{Q}/\mathbb{Z} & \rightarrow\mathbb{C}^{\times}\\
& b & \rightarrow\exp\left(  2\pi ib\right)  .
\end{array}
\]
Let $V=\mathbb{Q}_{p}^{n}$ be the $n-$dimensional $\mathbb{Q}_{p}-$vector
space and $V^{\prime}$ its algebraic dual vector space. We identify
$V^{\prime}$ with $\mathbb{Q}_{p}^{n}$ via the $\mathbb{Q}_{p}-$bilinear form%
\[
x\cdot y=x_{1}y_{1}+\ldots+x_{n}y_{n},
\]
$x\in V=\mathbb{Q}_{p}^{n}$, $y\in V^{\prime}=\mathbb{Q}_{p}^{n}$. Now we
identify $V^{\prime}$ with the topological dual $V^{\ast}$ of $V$ (i.e. the
group of all continuous additive characters on $\left(  V,+\right)  $) as
right $\mathbb{Q}_{p}-$vector spaces by means of the pairing $\Psi\left(
x\cdot y\right)  $. The Haar measure $dx$ \ is autodual with respect to this pairing.

We denote \ the space of all Schwartz--Bruhat functions on $\mathbb{Q}_{p}%
^{n}$ by $S:=S\left(  \mathbb{Q}_{p}^{n}\right)  $. For $\phi\in S$, we define
its Fourier transform $\mathcal{F}\phi$ by
\[
\left(  \mathcal{F}\phi\right)  \left(  \xi\right)  =%
{\textstyle\int\limits_{\mathbb{Q}_{p}^{n}}}
\Psi\left(  -x\cdot\xi\right)  \phi\left(  x\right)  d^{n}x.
\]
Then the Fourier transform induces a linear isomorphism of $S$\ onto $S$ and
the inverse Fourier transform is given by%
\[
\left(  \mathcal{F}^{-1}\varphi\right)  \left(  x\right)  =%
{\textstyle\int\limits_{\mathbb{Q}_{p}^{n}}}
\Psi\left(  x\cdot\xi\right)  \varphi\left(  \xi\right)  d^{n}\xi.
\]
The map $\phi\rightarrow\mathcal{F}\phi$ is an $L^{2}$-isometry on
$L^{1}\mathbb{\cap}L^{2}$, which is a dense subspace of $L^{2}$.

\textit{A }$p$\textit{-adic pseudo-differential operator} $f(D,\alpha)$, with
symbol $|f|_{p}^{\alpha}$, $f:\mathbb{Q}_{p}^{n}\rightarrow\mathbb{Q}_{p}$,
and $\alpha>0$, is an operator of the form
\[
f(D,\alpha)(\phi)=\mathcal{F}^{-1}(|f|_{p}^{\alpha}\mathcal{F}(\phi))
\]
for $\phi\in S$.

\section{\label{Sect1}Quasielliptic Polynomials}

In this section, motivated by \cite{G-V} and \cite{ZG3}, we introduce the
quasielliptic polynomials in a $p$-adic setting and give some basic properties
that we will use later on.

\begin{definition}
\label{def_qelliptic}Let $f\left(  \xi\right)  \in\mathbb{Q}_{p}\left[
\xi_{1},\ldots,\xi_{n}\right]  $ be a non-constant polynomial. Let
$\boldsymbol{w=}\left(  w_{1},\ldots,w_{n}\right)  \in\left(
\mathbb{N\smallsetminus}\left\{  0\right\}  \right)  ^{n}$. We say that
$f\left(  \xi\right)  $ is a quasielliptic polynomial of degree $d$ with
respect to $\boldsymbol{w}$, if the two following conditions hold:

\noindent(Q1) $f\left(  \lambda^{w_{1}}\xi_{1},\dots,\lambda^{w_{n}}\xi
_{n}\right)  =\lambda^{d}f\left(  \xi\right)  $, for $\lambda\in\mathbb{Q}%
_{p}^{\times}$,

\noindent(Q2) $f\left(  \xi_{1},\dots,\xi_{n}\right)  =0\Leftrightarrow\xi
_{1}=\dots=\xi_{n}=0$.
\end{definition}

Given $m\in\mathbb{N\smallsetminus}\left\{  0\right\}  $, we denote by
$(\mathbb{Q}_{p}^{\times})^{m}$ the multiplicative subgroup of $\mathbb{Q}%
_{p}^{\times}$ formed by the $m$-th powers.

\begin{remark}
\label{note1} (1) The elliptic quadratic forms are quasielliptic polynomials
of degree $2$ with respect to $\boldsymbol{w=}\left(  1,\ldots,1\right)  $,
see \cite{B-S} and \cite{Koch}. The elliptic polynomials introduced in
\cite{ZG3}\ are quasielliptic polynomials of degree $d$ with respect to
$\boldsymbol{w=}\left(  1,\ldots,1\right)  $.

(2) If $f\left(  \xi_{1},\ldots,\xi_{n}\right)  $ is a quasielliptic
polynomial of degree $d$ with respect to $\boldsymbol{w}$, then $\left(
f\left(  \xi_{1},\ldots,\xi_{n}\right)  \right)  ^{a}-$ $\tau\xi_{n+1}^{d}$,
with $g.c.d.\left(  a,d\right)  >1$ and $\tau\in\mathbb{Q}_{p}^{\times
}\smallsetminus\left(  \mathbb{Q}_{p}^{\times}\right)  ^{g.c.d.\left(
a,d\right)  }$, is a quasielliptic polynomial of degree $da$ with respect to
$\left(  w_{1},\ldots,w_{n},a\right)  $. In particular there exist infinitely
many quasielliptic polynomials.

(3) Note that if $f\left(  \xi\right)  $ is quasielliptic then $cf\left(
\xi\right)  ,$ $c\in\mathbb{Q}_{p}^{\times},$ is also quasielliptic. For this
reason, we will assume that all quasielliptic polynomials have coefficients in
$\mathbb{Z}_{p}$.
\end{remark}

From now on, we assume that $f\left(  \xi\right)  $ is a quasielliptic
polynomial of degree $d$ with respect to $\boldsymbol{w}$.

\begin{lemma}
\label{lemma1} Any quasielliptic polynomial $f\left(  \xi\right)  $ can be
written uniquely as
\[
f\left(  \xi\right)  ={\sum\limits_{i=1}^{n}}c_{i}\xi_{i}^{\frac{d}{w_{i}}%
}+h\left(  \xi\right)  ,
\]
where $c_{i}\in\mathbb{Z}_{p}\smallsetminus\left\{  0\right\}  $,
$i=1,\ldots,n$, and $h\left(  \xi\right)  \in\mathbb{Z}_{p}\left[  \xi
_{1},\ldots,\xi_{n}\right]  $, in addition, \ $h\left(  \xi\right)  $ does not
contain any monomial of the form $c_{i}\xi_{i}^{a_{i}}$, for $i=1,\ldots,n$.
\end{lemma}

\begin{proof}
By induction on $n$, the number of variables. The case $n=1$ is clear. Assume
by induction hypothesis the result for polynomials in $n\leq k$, $k\geq1$,
variables. Let $f\left(  \xi_{1},\ldots,\xi_{k},\xi_{k+1}\right)  $ be a
quasielliptic polynomial satisfying lemma's hypotheses. Note that
\[
f\left(  \xi_{1},\ldots,\xi_{k},\xi_{k+1}\right)  =\xi_{k+1}^{a}h_{0}\left(
\xi_{1},\ldots,\xi_{k},\xi_{k+1}\right)  +h_{1}\left(  \xi_{1},\ldots,\xi
_{k}\right)  ,
\]
for some polynomials $h_{0}\neq0$ and $h_{1}\neq0$, since $f(\xi_{1},\dots
,\xi_{k+1})$ is quasielliptic. In addition $h_{1}\left(  \xi_{1},\ldots
,\xi_{k}\right)  $ is a quasielliptic polynomial of degree $d$ with respect to
$\boldsymbol{w}^{\prime}\boldsymbol{=}\left(  w_{1},\ldots,w_{k}\right)  $.
Indeed, condition (Q1) is clear. To verify condition (Q2), we note that if
$h_{1}\left(  \xi_{1}^{\prime},\ldots,\xi_{k}^{\prime}\right)  =0$, then
$\left(  \xi_{1}^{\prime},\ldots,\xi_{k}^{\prime}\right)  $ satisfies
$f\left(  \xi_{1}^{\prime},\ldots,\xi_{k}^{\prime},0\right)  =0$ and since
$f(\xi_{1},\dots,\xi_{k+1})$ is quasielliptic, $\xi_{1}^{\prime}=\ldots
=\xi_{k}^{\prime}=0$. By applying the induction hypothesis to $h_{1}\left(
\xi_{1},\ldots,\xi_{k}\right)  $, one gets
\begin{equation}
f\left(  \xi_{1},\ldots,\xi_{k},\xi_{k+1}\right)  =\xi_{k+1}^{a}h_{0}\left(
\xi_{1},\ldots,\xi_{k},\xi_{k+1}\right)  +{\sum\limits_{i=1}^{k}}c_{i}\xi
_{i}^{\frac{d}{w_{i}}}+h_{2}\left(  \xi_{1},\ldots,\xi_{k}\right)  ,
\label{cond1}%
\end{equation}
where $c_{i}\in\mathbb{Z}_{p}\smallsetminus\left\{  0\right\}  $, $\frac
{d}{w_{i}}\in\mathbb{N\smallsetminus}\left\{  0\right\}  $ for $i=1,\ldots,k$,
and $h_{2}\left(  \xi\right)  \in\mathbb{Z}_{p}\left[  \xi_{1},\ldots,\xi
_{k}\right]  $.

We now note that $f\left(  0,\ldots,0,\xi_{k+1}\right)  \neq0$ is a
quasielliptic polynomial of degree $d$ with respect to $w_{k+1}$ and by the
induction hypothesis with $n=1$,
\begin{equation}
f\left(  0,\ldots,0,\xi_{k+1}\right)  =c_{k+1}\xi_{k+1}^{\frac{d}{w_{k+1}}%
}=\xi_{k+1}^{a}h_{0}\left(  0,\ldots,0,\xi_{k+1}\right)  . \label{cond2}%
\end{equation}
The result follows from (\ref{cond1})-(\ref{cond2}) by noting that%

\[
h_{0}\left(  \xi_{1},\ldots,\xi_{k},\xi_{k+1}\right)  =h_{0}\left(
0,\ldots,0,\xi_{k+1}\right)  +h_{3}\left(  \xi_{1},\ldots,\xi_{k},\xi
_{k+1}\right)  ,
\]
where $h_{3}\left(  \xi_{1},\ldots,\xi_{k},\xi_{k+1}\right)  $ does not
contain any monomial of the form $b_{k+1}\xi_{k+1}^{m_{k+1}}$, and that
\[
\xi_{k+1}^{a}h_{0}\left(  \xi_{1},\ldots,\xi_{k},\xi_{k+1}\right)  =c_{k+1}%
\xi_{k+1}^{\frac{d}{w_{k+1}}}+h_{4}(\xi_{1},\ldots,\xi_{k},\xi_{k+1}),
\]
where $c_{k+1}\in\mathbb{Z}_{p}\backslash\left\{  0\right\}  $ and
$h_{4}\left(  \xi\right)  \in\mathbb{Z}_{p}\left[  \xi_{1},\ldots,\xi_{k}%
,\xi_{k+1}\right]  $.
\end{proof}

\begin{definition}
Following Gindikin and Volevich \cite{G-V}, we attach to $f\left(  \xi\right)
$ the function
\[
\Xi\left(  \xi\right)  :={\sum\limits_{i=1}^{n}}\left\vert \xi_{i}\right\vert
_{p}^{\frac{d}{w_{i}}}.
\]

\end{definition}

\subsection{A basic estimate}

In this paragraph we establish some inequalities involving quasielliptic
polynomials that will play the central role in the next sections. We need some
preli\-minary results.

Set $T:=\left\{  -1,1\right\}  ^{n}$. Given $\boldsymbol{j}=(j_{1},\dots
,j_{n})\in T$, we define the $\boldsymbol{j}$\textit{-orthant of }%
$\mathbb{Q}_{p}^{n}$ to be
\[
\mathbb{Q}_{p}^{n}\left(  \boldsymbol{j}\right)  =\left\{  \left(  \xi
_{1},\dots,\xi_{n}\right)  \in\mathbb{Q}_{p}^{n};ord\left(  \xi_{i}\right)
<0\Leftrightarrow j_{i}=-1\right\}  .
\]
Note that if $\boldsymbol{j}=\left(  1,1,\ldots,1\right)  $, then
$\mathbb{Q}_{p}^{n}\left(  \boldsymbol{j}\right)  =\mathbb{Z}_{p}^{n}$.

Define for $l\in\mathbb{Z}$,
\[%
\begin{array}
[c]{llll}%
d\left(  l,\boldsymbol{j}\right)  : & \mathbb{Q}_{p}^{n}\left(  \boldsymbol{j}%
\right)  & \rightarrow & \mathbb{Q}_{p}^{n}\\
&  &  & \\
& \left(  \xi_{1},\ldots,\xi_{n}\right)  & \mapsto & \left(  \eta_{1}%
,\ldots,\eta_{n}\right)  ,
\end{array}
\]
where
\[
\eta_{i}=\left\{
\begin{array}
[c]{lll}%
p^{lw_{i}}\xi_{i} & \text{if} & j_{i}=-1,\\
&  & \\
\xi_{i} & \text{if} & j_{i}=1.
\end{array}
\right.
\]
Note that $d\left(  l,\boldsymbol{j}\right)  \circ d\left(  l^{\prime
},\boldsymbol{j}\right)  =d\left(  l+l^{\prime},\boldsymbol{j}\right)  $,
whenever $d(l^{\prime},\boldsymbol{j})(\mathbb{Q}_{p}^{n}(\boldsymbol{j}%
))\subseteq\mathbb{Q}_{p}^{n}(\boldsymbol{j}).$\newline For $\boldsymbol{j}%
\neq(1,\dots,1)$, set
\[
U_{n,\boldsymbol{j}}:=\left\{  \left(  \eta_{1},\ldots,\eta_{n}\right)
\in\mathbb{Z}_{p}^{n};\text{there exists }i\text{ such that }j_{i}=-1\text{
and }ord\left(  \eta_{i}\right)  \leq w_{i}-1\right\}  .
\]

\begin{lemma}
\label{lemma2}Let $\boldsymbol{j\in}T$ such that $\boldsymbol{j}\neq\left(
1,1,\ldots,1\right)  $, then the following assertions hold:

\noindent(1) $\mathbb{Q}_{p}^{n}\left(  \boldsymbol{j}\right)  \subset
{\bigsqcup\limits_{l=1}^{\infty}}d\left(  -l,\boldsymbol{j}\right)
U_{n,\boldsymbol{j}}$;

\noindent(2) $U_{n,\boldsymbol{j}}$ is a compact subset of $\mathbb{Z}_{p}%
^{n}$.
\end{lemma}

\begin{proof}
(1) By renaming the variables, we may assume that
\[
\boldsymbol{j=}\left(  -1,\ldots,\underbrace{-1}_{r-\text{th place}}%
,1,\ldots,1\right)  ,\text{ with }1\leq r\leq n.
\]
Let $\xi=\left(  \xi_{1},\ldots,\xi_{n}\right)  \in\mathbb{Q}_{p}^{n}\left(
\boldsymbol{j}\right)  $, with $\xi_{i}=p^{m_{i}}\upsilon_{i} $, $m_{i}%
\in\mathbb{Z}$, $m_{i}<0$, $\upsilon_{i}\in\mathbb{Z}_{p}^{\times}$, for
$i=1,\ldots,r$ and $\xi_{i}\in\mathbb{Z}_{p}$, for $i=r+1,\ldots,n$. Set
$m_{i}=q_{i}w_{i}+r_{i}$, with $0\leq r_{i}\leq w_{i}-1 $, for $i=1,\dots,r$.
Note that the $q_{i}$'s are negative.

\textbf{Claim A. }If $l_{\xi}:=\max_{i}\left\{  -q_{i};1\leq i\leq r\right\}
$, then $d\left(  l_{\xi},\boldsymbol{j}\right)  \xi\in U_{n,\boldsymbol{j}}$.

Since $l_{\xi}w_{i}+m_{i}\geq0$, $p^{l_{\xi}w_{i}}\xi_{i}\in\mathbb{Z}_{p}$
for $i=1,\ldots,r$, then
\[
d(l_{\xi},\boldsymbol{j})\xi=(p^{l_{\xi}w_{1}}\xi_{1},\dots,p^{l_{\xi}w_{r}%
}\xi_{r},\xi_{r+1},\dots,\xi_{n})\in\mathbb{Z}_{p}^{n}.
\]
In addition, there exists $i_{0}$ such that $l_{\xi}=-q_{i_{0}}$ and
\[
ord(p^{l_{\xi}w_{i_{0}}}\xi_{i_{0}})=-q_{i_{0}}w_{i_{0}}+m_{i_{0}}=r_{i_{0}%
}\leq w_{i_{0}}-1.
\]
From Claim A follows that $\mathbb{Q}_{p}^{n}\left(  \boldsymbol{j}\right)
\subset{\bigcup\limits_{l=1}^{\infty}}d\left(  -l,\boldsymbol{j}\right)
U_{n,\boldsymbol{j}}$. To establish (1) it is sufficient to show that

\textbf{Claim B.}
\[
D\left(  l,l^{\prime},\boldsymbol{j}\right)  :=d\left(  -l,\boldsymbol{j}%
\right)  U_{n,\boldsymbol{j}}\cap d\left(  -l^{\prime},\boldsymbol{j}\right)
U_{n,\boldsymbol{j}}=\emptyset\text{, if }l\neq l^{\prime}\text{.}%
\]
Assume that $l^{\prime}-l<0$ and that $\xi=\left(  \xi_{1},\ldots,\xi
_{n}\right)  \in D\left(  l,l^{\prime},\boldsymbol{j}\right)  $. Then
$\xi=d\left(  -l,\boldsymbol{j}\right)  \eta=d\left(  -l^{\prime
},\boldsymbol{j}\right)  \eta^{\prime}$, for some $\eta$, $\eta^{\prime}\in
U_{n,\boldsymbol{j}}$ and $d\left(  l^{\prime}-l,\boldsymbol{j}\right)
\eta=\eta^{\prime}$, therefore $p^{\left(  l^{\prime}-l\right)  w_{i}}\eta
_{i}=\eta_{i}^{\prime}$ for $j_{i}=-1$. Since there exists $i_{0}$, with
$j_{i_{0}}=-1$, such that $ord\left(  \eta_{i_{0}}\right)  \leq w_{i_{0}}-1$,
then $\eta_{i_{0}}^{\prime}=p^{(l^{\prime}-l)w_{i_{0}}}\eta_{i_{0}}$ and
\begin{align*}
ord(\eta_{i_{0}}^{\prime})  &  =(l^{\prime}-l)w_{i_{0}}+ord(\eta_{i_{0}})\\
&  \leq(l^{\prime}-l)w_{i_{0}}+w_{i_{0}}-1\\
&  \leq w_{i_{0}}(l^{\prime}-l+1)-1<0,
\end{align*}
contradicting $\eta_{i_{0}}^{\prime}\in\mathbb{Z}_{p}$.

(2) Note that
\[
U_{n,\boldsymbol{j}}=\bigsqcup_{I\subset\{1,\dots,n\},I\neq\emptyset}\{\eta
\in\mathbb{Z}_{p}^{n};ord(\eta_{i})\leq w_{i}-1\text{ and }j_{i}%
=-1\Leftrightarrow i\in I\}.
\]

By renaming the variables, we may assume $I=\{1,\dots,r\}$,$\ 1\leq r\leq n$,
then
\[
\{\eta\in\mathbb{Z}_{p}^{n};ord(\eta_{i})\leq w_{i}-1\text{ and }%
j_{i}=-1\Leftrightarrow i\in I\}=\prod_{i=1}^{r}p^{w_{i}-1}\mathbb{Z}%
_{p}\times(\mathbb{Z}_{p})^{n-r},
\]
which is a compact set, therefore $U_{n,\boldsymbol{j}}$ is compact subset of
$\mathbb{Z}_{p}^{n}$.
\end{proof}

Set
\[
V_{n,\boldsymbol{j}}:=\{\eta\in\mathbb{Z}_{p}^{n};\text{ there exists an index
}i\text{ such that }ord(\eta_{i})\leq w_{i}-1\}.
\]
Set $\boldsymbol{j}=(1,1,\dots,1)$. Note that $\mathbb{Q}_{p}^{n}\left(
\boldsymbol{j}\right)  =\mathbb{Z}_{p}^{n}$. For $l\in\mathbb{N}$, define
\[%
\begin{array}
[c]{llll}%
d\left(  l,\boldsymbol{j}\right)  : & \mathbb{Z}_{p}^{n} & \rightarrow &
\mathbb{Z}_{p}^{n}\\
&  &  & \\
& \left(  \xi_{1},\ldots,\xi_{n}\right)  & \mapsto & \left(  \eta_{1}%
,\ldots,\eta_{n}\right)  ,
\end{array}
\]
where $\eta_{i}=p^{lw_{i}}\xi_{i},$ for $i=1,\dots,n$.

\begin{lemma}
\label{lemma2parte2}Set $\boldsymbol{j}=\left(  1,1,\ldots,1\right)  $, then
the following assertions hold:

\noindent(1) $\mathbb{Z}_{p}^{n}=\left\{  0\right\}  \cup{\bigsqcup
\limits_{l=0}^{\infty}}d\left(  l,\boldsymbol{j}\right)  V_{n,\boldsymbol{j}}$;

\noindent(2) $V_{n,\boldsymbol{j}}$ is a compact subset of $\mathbb{Z}_{p}%
^{n}$.
\end{lemma}

\begin{proof}
(1)-(2) Let $\xi=(\xi_{1},\dots,\xi_{n})\in\mathbb{Z}_{p}^{n}$, with
$\xi=p^{m_{i}}v_{i},\ m_{i}\geq0,\ v_{i}\in\mathbb{Z}_{p}^{\times}$ for
$i=1,\dots,n$. Set $m_{i}=q_{i}w_{i}+r_{i}$, with $0\leq r_{i}\leq w_{i}-1$,
$i=1,\dots,n$, and $l_{\xi}:=\min\{q_{i}:1\leq i\leq n\}$. The result follows
by using the reasoning given in the proof of Lemma \ref{lemma2}.
\end{proof}

\begin{lemma}
\label{lemma3}(1) Let $g\left(  \xi\right)  \in\mathbb{Z}_{p}\left[  \xi
_{1},\ldots,\xi_{n}\right]  $ be a non-constant polynomial satisfying
\begin{equation}
g\left(  \xi\right)  =0\Leftrightarrow\xi=0. \label{hyp1}%
\end{equation}
Let $A\subseteq\mathbb{Q}_{p}^{n}$ be a compact subset such that \ $0\notin
A$. There exists $M=M\left(  A,g\right)  \in\mathbb{N\smallsetminus}\left\{
0\right\}  $ such that
\[
\left\vert g\left(  \xi\right)  \right\vert _{p}\geq p^{-M}\text{, for any
}\xi\in A\text{.}%
\]

(2) There exist a constant $R\in\mathbb{N}$, and a finite number of points
$\widetilde{\xi}_{i}\in A$ such that if $B_{M+1+R}\left(  \widetilde{\xi}%
_{i}\right)  :=\widetilde{\xi}_{i}+\left(  p^{M+1+R}\mathbb{Z}_{p}\right)
^{n}$, then $A=\cup_{i}B_{M+1+R}\left(  \widetilde{\xi}_{i}\right)  $ and
\[
\left\vert g\left(  \xi\right)  \right\vert _{p}\mid_{B_{M+1+R}\left(
\widetilde{\xi}_{i}\right)  }=\left\vert g\left(  \widetilde{\xi}_{i}\right)
\right\vert _{p},
\]
for every $i$.
\end{lemma}

\begin{proof}
(1) By contradiction, assume the existence of a sequence $\xi^{\left(
n\right)  }$ of points of $A$ such that $\left\vert g\left(  \xi^{\left(
n\right)  }\right)  \right\vert _{p}\leq p^{-n}$, for $n\in\mathbb{N}$. Since
$A$ is compact, by passing to a subsequence if necessary, we have
$\xi^{\left(  n\right)  }\rightarrow\widetilde{\xi}\in A$, and \ thus
$\widetilde{\xi}\neq0$. On other hand, by continuity $g\left(  \widetilde{\xi
}\right)  =0$, contradicting (\ref{hyp1}).

(2) Let $\widetilde{\xi}_{i}\in A$ be a given point. Then%

\[
g\left(  \widetilde{\xi}_{i}+p^{M+1}\eta\right)  =g\left(  \widetilde{\xi}%
_{i}\right)  +p^{M+1}h_{i}\left(  \eta\right)  ,
\]
where $h_{i}\left(  \eta\right)  \in\mathbb{Q}_{p}\left[  \eta\right]  $
satisfying $h_{i}\left(  0\right)  =0$. By continuity, there exists $R_{i}%
\in\mathbb{N}$, such that if $\eta\in\left(  p^{R_{i}}\mathbb{Z}_{p}\right)
^{n}$ then $\left\vert h_{i}\left(  \eta\right)  \right\vert _{p}\leq1$. Hence
for any natural number $R\geq R_{i}$, we have
\[
\left\vert g\left(  \widetilde{\xi}_{i}+p^{M+1}\eta\right)  \right\vert
_{p}=\left\vert g\left(  \widetilde{\xi}_{i}\right)  +p^{M+1}h_{i}\left(
\eta\right)  \right\vert _{p}=\left\vert g\left(  \widetilde{\xi}_{i}\right)
\right\vert _{p}\text{,}%
\]
for any $\eta\in\left(  p^{R}\mathbb{Z}_{p}\right)  ^{n}\subset\left(
p^{R_{i}}\mathbb{Z}_{p}\right)  ^{n}$. Since $A$ is compact, there exists a
finite number of points $\widetilde{\xi}_{i}\in A$, such that $\widetilde{\xi
}_{i}+\left(  p^{M+1+R}\mathbb{Z}_{p}\right)  ^{n}$ for $i=1,\ldots,k$, is a
covering of $A$. We take $R:=\max_{1\leq i\leq k}R_{i}$.
\end{proof}

\begin{proposition}
\label{Prop1} There exist positive constants $A_{0}$, $A_{1}$ such that
\begin{equation}
A_{0}\Xi(\xi)\leq|f(\xi)|_{p}\leq A_{1}\Xi(\xi),\text{ for }||\xi||_{p}\geq p.
\label{basicineq}%
\end{equation}

\end{proposition}

\begin{proof}
Since $\left\{  \xi\in\mathbb{Q}_{p}^{n};||\xi||_{p}\geq p\right\}
\subseteq\bigsqcup\nolimits_{\boldsymbol{j}\neq(1,1,\ldots,1)}\mathbb{Q}%
_{p}^{n}(\boldsymbol{j})$, by Lemma \ref{lemma2}, it is sufficient to show
inequality (\ref{basicineq}) for $\ \xi\in d\left(  -l,\boldsymbol{j}\right)
U_{n,\boldsymbol{j}}$, for $l\in\mathbb{N}$. By renaming the variables we may
assume that
\[
\boldsymbol{j=}\left(  -1,\ldots,\underbrace{-1}_{r-\text{th place}}%
,1,\ldots,1\right)  ,\text{ with }1\leq r\leq n,
\]
and that $\xi=(\xi_{1},\dots,\xi_{n})=d\left(  -l,\boldsymbol{j}\right)  \eta
$, $\eta=(\eta_{1},\dots,\eta_{n})\in U_{n,\boldsymbol{j}}$, for some
$l\in\mathbb{N}$, i.e.
\[
\xi_{i}=\left\{
\begin{array}
[c]{lll}%
p^{-lw_{i}}\eta_{i} & \eta_{i}\in\mathbb{Z}_{p}, & 1\leq i\leq r,\\
&  & \\
\eta_{i} & \eta_{i}\in\mathbb{Z}_{p}, & r+1\leq i\leq n,
\end{array}
\right.
\]
and, there exists an index $i_{0}$ such that $\eta_{i_{0}}=p^{\left(
w_{i_{0}}-1\right)  }\upsilon$, $\upsilon\in\mathbb{Z}_{p}^{\times}$.

If $r<n$, we write $f\left(  \xi\right)  =f_{\boldsymbol{j}}\left(  \xi
_{1},\ldots,\xi_{r}\right)  +t\left(  \xi_{1},\ldots,\xi_{n}\right)  $, where
$f_{\boldsymbol{j}}\left(  \xi_{1},\ldots,\xi_{r}\right)  $ is a quasielliptic
polynomial of degree $d$ with respect to $\boldsymbol{w}^{\prime
}\boldsymbol{=}\left(  w_{1},\ldots,w_{r}\right)  $. By Lemma \ref{lemma1},
$f_{\boldsymbol{j}}\left(  \xi\right)  =\sum_{i=1}^{r}c_{i}\xi_{i}^{\frac
{d}{w_{i}}}+h\left(  \xi\right)  $, therefore
\begin{equation}
\left\vert f\left(  \xi\right)  \right\vert _{p}=\left\vert p^{-ld}%
f_{\boldsymbol{j}}\left(  \eta\right)  +p^{-la}t_{1}\left(  \eta\right)
\right\vert _{p}, \label{rel0}%
\end{equation}
with $t_{1}\left(  \eta\right)  \in\mathbb{Z}_{p}\left[  \eta\right]  $,
$a<d$, since $r<n$. By Lemma \ref{lemma3}, if $l>\frac{M}{d-a}$, then
$\left\vert f\left(  \xi\right)  \right\vert _{p}=p^{ld}\left\vert
f_{\boldsymbol{j}}\left(  \eta\right)  \right\vert _{p}$ and
\[
p^{ld}\inf_{\eta\in U_{n,\boldsymbol{j}}}\left\vert f_{\boldsymbol{j}}\left(
\eta\right)  \right\vert _{p}\leq\left\vert f\left(  \xi\right)  \right\vert
_{p}\leq p^{ld}\sup_{\eta\in U_{n,\boldsymbol{j}}}\left\vert f_{\boldsymbol{j}%
}\left(  \eta\right)  \right\vert _{p}.
\]

Since $||\xi||_{p}\geq p^{M_{0}}$, with $M_{0}\in\mathbb{N}\smallsetminus
\left\{  0\right\}  $ satisfying $M_{0}>\frac{M(w_{1}+\cdots+w_{n})}{(d-a)}$,
implies that $l>\frac{M}{d-a}$, the previous inequality can be re-written as
\begin{equation}
p^{ld}\inf_{\eta\in U_{n,\boldsymbol{j}}}\left\vert f_{\boldsymbol{j}}\left(
\eta\right)  \right\vert _{p}\leq\left\vert f\left(  \xi\right)  \right\vert
_{p}\leq p^{ld}\sup_{\eta\in U_{n,\boldsymbol{j}}}\left\vert f_{\boldsymbol{j}%
}\left(  \eta\right)  \right\vert _{p},\text{ for }\left\Vert \xi\right\Vert
_{p}\geq p^{M_{0}}. \label{ineq1}%
\end{equation}

Note that by Lemma \ref{lemma3}, $\inf_{\eta\in U_{n,\boldsymbol{j}}%
}\left\vert f_{\boldsymbol{j}}\left(  \eta\right)  \right\vert _{p}$ and
$\sup_{\eta\in U_{n,\boldsymbol{j}}}\left\vert f_{\boldsymbol{j}}\left(
\eta\right)  \right\vert _{p}$ are positive numbers.

We now take $\xi\in\mathbb{Q}_{p}^{n}(\boldsymbol{j})$ satisfying $||\xi
||_{p}=p^{l_{0}}$ with $1\leq l_{0}<M_{0}$. By (\ref{rel0}) with $l=l_{0}$ one
gets $\left\vert f\left(  \xi\right)  \right\vert _{p}=p^{l_{0}d}\left\vert
g_{\boldsymbol{j}},_{l_{0}}\left(  \eta\right)  \right\vert _{p}$, and by
applying Lemma \ref{lemma3} to $g_{\boldsymbol{j}},_{l_{0}}\left(
\eta\right)  $ and $U_{n,\boldsymbol{j}}$,%
\begin{equation}
C_{1}\left(  l_{0}\right)  p^{l_{0}d}\leq\left\vert f\left(  \xi\right)
\right\vert _{p}\leq C_{2}\left(  l_{0}\right)  p^{l_{0}d}, \label{ineq1A}%
\end{equation}
for some positive constants $C_{1}\left(  l_{0}\right)  $, $C_{2}\left(
l_{0}\right)  $. By combining (\ref{ineq1})-(\ref{ineq1A}), we have%
\begin{equation}
C_{1}p^{ld}\leq\left\vert f\left(  \xi\right)  \right\vert _{p}\leq
C_{2}p^{ld}, \label{ineq2}%
\end{equation}
for some positive constants $C_{1}$, $C_{2}$.

If $r=n$, then $\xi_{i}=p^{-lw_{i}}\eta_{i}$,$\ \eta_{i}\in\mathbb{Z}$ for
$1\leq i\leq n$ and $|f(\xi)|_{p}=p^{ld}|f(\eta)|_{p}$, and thus inequality
(\ref{ineq2}) holds.

On the other hand,
\[
\Xi\left(  \xi\right)  =p^{ld}{\sum\limits_{i=1}^{r}}\left\vert \eta
_{i}\right\vert _{p}^{\frac{d}{w_{i}}}+{\sum\limits_{i=r+1}^{n}}\left\vert
\eta_{i}\right\vert _{p}^{\frac{d}{w_{i}}}\text{, for }1\leq r\leq n\text{,}%
\]
therefore
\begin{equation}
\Xi\left(  \xi\right)  \leq np^{ld}, \label{ximenor}%
\end{equation}
and
\begin{align}
\Xi\left(  \xi\right)   &  \geq p^{ld}\left(  {\sum\limits_{\left\{
i;ord\left(  \eta_{i}\right)  \leq w_{i}-1\right\}  }}\left\vert \eta
_{i}\right\vert _{p}^{\frac{d}{w_{i}}}\right) \label{ximayor}\\
&  \geq p^{ld}\min_{\eta\in U_{n,\boldsymbol{j}}}\left(  \sum_{\{i:ord(\eta
_{i})\leq w_{i}-1\}}p^{-\frac{(w_{i}-1)d}{w_{i}}}\right)  \geq C_{1}%
p^{ld},\nonumber
\end{align}
where $C_{1}$ is a positive constant. The result follows from inequalities
(\ref{ineq2})-(\ref{ximayor1}).
\end{proof}

\begin{proposition}
\label{Prop1parte2} There exist positive constants $A_{0}$, $A_{1}$, such
that
\begin{equation}
A_{0}\Xi\left(  \xi\right)  \leq\left\vert f\left(  \xi\right)  \right\vert
_{p}\leq A_{1}\Xi\left(  \xi\right)  , \text{ for }\left\Vert \xi\right\Vert
_{p}\leq1. \label{basicineq2}%
\end{equation}

\end{proposition}

\begin{proof}
Set $\boldsymbol{j}=(1,1,\dots,1)$. Since $\mathbb{Z}_{p}^{n}=\sqcup
_{l=0}^{\infty}d(l,\boldsymbol{j})V_{n,\boldsymbol{j}}$, it is sufficient to
show inequality (\ref{basicineq2}) for $\xi\in d(l,\boldsymbol{j}%
)V_{n,\boldsymbol{j}}$ for $l\in\mathbb{N}$.

Set $\xi_{i}=p^{lw_{i}}\eta_{i}$ for $i=1,\dots,n$, with $\eta_{i}\in
V_{n,\boldsymbol{j}}$, since $f$ is quasielliptic
\[
f(\xi)=f(p^{lw_{1}}\eta_{1},\dots,p^{lw_{n}}\eta_{n})=p^{ld}f(\eta),
\]
and $|f(\xi)|_{p}=p^{-ld}|f(\eta)|_{p}$. By using the fact that
$V_{n,\boldsymbol{j}}$ is compact and $0\notin V_{n,\boldsymbol{j}}$ we have
\[
p^{-ld}\inf_{\eta\in V_{n,\boldsymbol{j}}}|f(\eta)|_{p}\leq|f(\xi)|_{p}\leq
p^{-ld}\sup_{\eta\in V_{n,\boldsymbol{j}}}|f(\eta)|_{p},
\]
with $\inf_{\eta\in V_{n,\boldsymbol{j}}}|f(\eta)|_{p}>0$.

On the other hand,
\begin{equation}
\Xi(\xi)=p^{-ld}\sum_{i=1}^{n}|\eta_{i}|_{p}^{\frac{d}{w_{i}}}\geq
p^{-ld}C_{2}, \label{ximayor1}%
\end{equation}
where $C_{2}:=\inf_{\eta\in V_{n,\boldsymbol{j}}}\left(  \sum_{i=1}^{n}%
|\eta_{i}|_{p}^{\frac{d}{w_{i}}}\right)  >0$. Since $\eta\in\mathbb{Z}_{p}%
^{n}$, we have
\begin{equation}
\Xi(\xi)\leq np^{-ld}. \label{ximenor1}%
\end{equation}
Then
\[
\frac{\Xi(\xi)}{n}\inf_{\eta\in V_{n,\boldsymbol{j}}}|f(\eta)|_{p}\leq
|f(\xi)|_{p}\leq\frac{\Xi(\xi)}{C_{2}}\sup_{\eta\in V_{n,\boldsymbol{j}}%
}|f(\eta)|_{p}.
\]

\end{proof}

As a consequence of Propositions \ref{Prop1}-\ref{Prop1parte2}, we obtain the
following result:

\begin{theorem}
\label{Theo1} Let $f(\xi)$ be a quasielliptic polynomial, then there exist
positive constants $A_{0},A_{1}$ such that
\begin{equation}
A_{0}\Xi(\xi)\leq|f(\xi)|_{p}\leq A_{1}\Xi(\xi),\text{\ for any }\xi
\in\mathbb{Q}_{p}^{n}. \label{inequalitytheorem}%
\end{equation}

\end{theorem}

\section{\label{Sect2}Semi-quasielliptic polynomials}

Let $\left\langle \cdot,\cdot\right\rangle $ denote the usual inner product of
$\mathbb{R}^{n}$.

\begin{definition}
\label{def_sqelliptic}A polynomial of the form
\begin{equation}
F\left(  \xi,x\right)  :=f\left(  \xi\right)  +{\sum\limits_{\left\langle
\boldsymbol{k},\boldsymbol{w}\right\rangle \leq d-1}}c_{\boldsymbol{k}}\left(
x\right)  \xi^{\boldsymbol{k}}\in\mathbb{Q}_{p}\left[  \xi_{1},\ldots,\xi
_{n}\right]  , \label{QS_symbol}%
\end{equation}
where $f(\xi)$ is a quasielliptic polynomial of degree $d$ with respect to
$\boldsymbol{w}$, and each $c_{\boldsymbol{k}}\left(  x\right)  :$
$\mathbb{Q}_{p}^{n}$ $\rightarrow$ $\mathbb{Q}_{p}$ satisfies
$||c_{\boldsymbol{k}}(x)||_{L^{\infty}}:=\sup_{x\in\mathbb{Q}_{p}^{n}%
}\left\vert c_{\boldsymbol{k}}\left(  x\right)  \right\vert _{p}<\infty$, is a
called a semi-quasielliptic polynomial with variable coefficients, or simply a
semi-quasielliptic polynomial.
\end{definition}

\begin{theorem}
\label{Teo3} Let $F(\xi,x)$ be a semi-quasielliptic polynomial, then there
exist positive constants, $A_{0}$, $A_{1}$, $M_{0}$ ($M_{0}\in\mathbb{N)}$,
which do not depend on $x$, such that
\begin{equation}
A_{0}\Xi(\xi)\leq|F\left(  \xi,x\right)  |_{p}\leq A_{1}\Xi(\xi),\ \text{ for
}||\xi||_{p}\geq p^{M_{0}}. \label{QS_inequality}%
\end{equation}

\end{theorem}

\begin{proof}
We use all the notation introduced in the proof of Proposition \ref{Prop1}.
Without loss of generality we assume that $\boldsymbol{j=}\left(
-1,\ldots,\underbrace{-1}_{r-\text{th place}},1,\ldots,1\right)  $, with
$1\leq r\leq n $, and set $\xi\in\mathbb{Q}_{p}^{n}(\boldsymbol{j})$.

If $r<n$, we write $F(\xi,x)=f_{\boldsymbol{j}}\left(  \xi_{1},\ldots,\xi
_{r}\right)  +t\left(  \xi_{1},\ldots,\xi_{n}\right)  +\sum_{\boldsymbol{k}%
}c_{\boldsymbol{k}}(x)\xi^{\boldsymbol{k}}$, then
\begin{align*}
|F(\xi,x)|_{p}  &  =|p^{-ld}f_{\boldsymbol{j}}(\eta)+p^{-la}t_{1}(\eta
)+\sum_{\boldsymbol{k}}p^{-l\delta_{\boldsymbol{k}}}\eta^{\boldsymbol{k}%
}c_{\boldsymbol{k}}(x)|_{p}\\
&  =p^{ld}|f_{\boldsymbol{j}}(\eta)+p^{l(d-a)}t_{1}(\eta)+\sum_{\boldsymbol{k}%
}p^{l(d-\delta_{\boldsymbol{k}})}\eta^{\boldsymbol{k}}c_{\boldsymbol{k}%
}(x)|_{p}%
\end{align*}
with $a<d$, since $r<n$, and $\delta_{\boldsymbol{k}}=\sum\nolimits_{i=1}%
^{r}k_{i}w_{i}\leq d-1<d$. By Lemma \ref{lemma3}, if $l>\max\left(  \left\{
\frac{M}{d-a}\right\}  \cup\cup_{\boldsymbol{k}}\left\{  \frac{M+\log
_{p}||c(x)||_{L^{\infty}}}{d-\delta_{\boldsymbol{k}}}\right\}  \right)  $,
then $|F(\xi,x)|_{p}=p^{ld}|f_{\boldsymbol{j}}(\eta)|_{p}$ and
\[
p^{ld}\inf_{\eta\in U_{n,\boldsymbol{j}}}\left\vert f_{\boldsymbol{j}}\left(
\eta\right)  \right\vert _{p}\leq\left\vert F\left(  \xi,x\right)  \right\vert
_{p}\leq p^{ld}\sup_{\eta\in U_{n,\boldsymbol{j}}}\left\vert f_{\boldsymbol{j}%
}\left(  \eta\right)  \right\vert _{p}.
\]
By using inequalities (\ref{ximenor1}) and (\ref{ximayor1}), we conclude
\[
\frac{\Xi(\xi)}{n}\inf_{\eta\in U_{n,\boldsymbol{j}}}|f_{\boldsymbol{j}}%
(\eta)|_{p}\leq|F(\xi,x)|_{p}\leq\frac{\Xi(\xi)}{C_{1}}\sup_{\eta\in
U_{n,\boldsymbol{j}}}|f_{\boldsymbol{j}}(\eta)|_{p}.
\]

On the other hand, if $||\xi||_{p}\geq p^{M_{0}}$, and
\[
M_{0}>(w_{1}+\cdots+w_{n})\max\left(  \left\{  \frac{M}{d-a}\right\}  \cup
\cup_{\boldsymbol{k}}\left\{  \frac{M+\log_{p}||c(x)||_{L^{\infty}}}%
{d-\delta_{\boldsymbol{k}}}\right\}  \right)  ,
\]
since $p^{l(w_{1}+\cdots+w_{n})}\geq||\xi||_{p}$, then we have $l>\max\left(
\left\{  \frac{M}{d-a}\right\}  \cup\cup_{\boldsymbol{k}}\left\{  \frac
{M+\log_{p}||c(x)||_{L^{\infty}}}{d-\delta_{\boldsymbol{k}}}\right\}  \right)
$, and
\begin{equation}
\frac{\Xi(\xi)}{n}\inf_{\eta\in U_{n,\boldsymbol{j}}}|f_{\boldsymbol{j}}%
(\eta)|_{p}\leq|F(\xi,x)|_{p}\leq\frac{\Xi(\xi)}{C_{1}}\sup_{\eta\in
U_{n,\boldsymbol{j}}}|f_{\boldsymbol{j}}(\eta)|_{p},\text{ for }||\xi
||_{p}\geq p^{M_{0}}. \label{ineq4}%
\end{equation}

If $r=n$, then
\[
|F(\xi,x)|_{p}=|p^{-ld}f_{\boldsymbol{j}}(\eta)+\sum_{\boldsymbol{k}%
}p^{-l\langle\boldsymbol{k},\boldsymbol{w}\rangle}\eta^{\boldsymbol{k}%
}c_{\boldsymbol{k}}(x)|_{p}=\left\vert p^{-ld}f_{\boldsymbol{j}}%
(\eta)\right\vert _{p},
\]
for $l>\max\left(  \cup_{\boldsymbol{k}}\left\{  \frac{M+\log_{p}%
||c_{\boldsymbol{k}}(x)||_{L^{\infty}}}{d-\langle\boldsymbol{k},\boldsymbol{w}%
\rangle}\right\}  \right)  $. If we choose $||\xi||_{p}\geq p^{M_{0}}$ with
$M_{0}>(w_{1}+\cdots+w_{n})\max\left(  \cup_{\boldsymbol{k}}\left\{
\frac{M+\log_{p}||c_{\boldsymbol{k}}(x)||_{L^{\infty}}}{d-\langle
\boldsymbol{k},\boldsymbol{w}\rangle}\right\}  \right)  $, we have
$l>\max\left(  \cup_{\boldsymbol{k}}\left\{  \frac{M+\log_{p}%
||c_{\boldsymbol{k}}(x)||_{L^{\infty}}}{d-\langle\boldsymbol{k},\boldsymbol{w}%
\rangle}\right\}  \right)  $ and
\begin{equation}
\frac{\Xi(\xi)}{n}\inf_{\eta\in U_{n,\boldsymbol{j}}}|f_{\boldsymbol{j}}%
(\eta)|_{p}\leq|F(\xi,x)|_{p}\leq\frac{\Xi(\xi)}{C_{1}}\sup_{\eta\in
U_{n,\boldsymbol{j}}}|f_{\boldsymbol{j}}(\eta)|_{p},\text{ for }||\xi
||_{p}\geq p^{M_{0}}. \label{ineq5}%
\end{equation}

The result follows from inequalities (\ref{ineq4})-(\ref{ineq5}).
\end{proof}

\section{\label{Sect3}Pseudo-differential operators with quasielliptic
symbols}

For $\alpha\geq0$, we set
\[
(\boldsymbol{f}(D;\alpha)\phi)(x):=\mathcal{F}_{\xi\rightarrow x}^{-1}\left(
|f(\xi)|_{p}^{\alpha}\mathcal{F}_{x\rightarrow\xi}(\phi)\right)  ,
\]
where $\alpha>0$, $\phi$ is a Lizorkin type function, and $f(\xi)$ is a
quasielliptic polynomial of degree $d$ with respect to $\boldsymbol{w}$. We
call an extension of $\boldsymbol{f}(D;\alpha)$ a\textit{\ pseudo-differential
operator with quasielliptic symbol}.

In this section we study equations of type $\boldsymbol{f}(D;\alpha)u=v$,
indeed we determine the functions spaces on which a such equation has a solution.

\subsection{Sobolev-type spaces}

\begin{definition}
\label{def_space}Given $\beta\geq0$ and $\Xi\left(  \xi\right)  $ as before,
we define for $\phi\in S$ the following norm:
\[
||\phi||_{(\beta,\Xi)}^{2}:=||\phi||_{\beta}^{2}=\int_{\mathbb{Q}_{p}^{n}}%
\Xi^{2\beta}(\xi)|\mathcal{F}(\phi)(\xi)|^{2}d^{n}\xi.
\]
Set
\[
{\LARGE \Phi}:=\{\phi\in S;\mathcal{F}(\phi)(0)=0\},
\]
and $\widetilde{H}_{(\beta,\Xi)}(\mathbb{Q}_{p}^{n}):=\widetilde{H}_{\beta}$
as the completion of $({\LARGE \Phi},||\cdot||_{\beta})$.
\end{definition}

\begin{remark}
The space ${\LARGE \Phi}$ is the Lizorkin space of the second kind introduced
in \cite{A-K-S}.
\end{remark}

\begin{lemma}
\label{lemma5} Set $\beta\geq\alpha$. The operator
\[%
\begin{array}
[c]{llll}%
{\boldsymbol{f}}(D;\alpha): & ({\LARGE \Phi},||\cdot||_{\beta}) & \rightarrow
& ({\LARGE \Phi},||\cdot||_{\beta-\alpha})\\
&  &  & \\
& \phi & \mapsto & \boldsymbol{f}(D;\alpha)\phi,
\end{array}
\]
is well-defined and continuous.
\end{lemma}

\begin{proof}
Note that $\phi\in{\LARGE \Phi}$ implies that $\boldsymbol{f}(D;\alpha)\phi
\in{\LARGE \Phi}$. The continuity of the operator follows from Theorem
\ref{Theo1}:%
\begin{align*}
||\boldsymbol{f}(D;\alpha)\phi||_{\beta-\alpha}^{2}  &  =\int_{\mathbb{Q}%
_{p}^{n}}\Xi^{2(\beta-\alpha)}(\xi)|f(\xi)|_{p}^{2\alpha}|\mathcal{F}%
(\phi)(\xi)|^{2}d^{n}\xi\\
&  \leq A_{1}\int_{\mathbb{Q}_{p}^{n}}\Xi^{2(\beta-\alpha)}(\xi)\ \Xi
^{2\alpha}(\xi)|\mathcal{F}(\phi)(\xi)|^{2}d^{n}\xi\\
&  =A_{1}\int_{\mathbb{Q}_{p}^{n}}\Xi^{2\beta}(\xi)|\mathcal{F}(\phi
)(\xi)|^{2}d^{n}\xi=A_{1}||\phi||_{\beta}^{2}.
\end{align*}

\end{proof}

By density, we extend $\boldsymbol{f}(D;\alpha)$ to $\overline{({\LARGE \Phi
},||\cdot||_{\beta})}=\widetilde{H}_{\beta}$.

\begin{lemma}
\label{lemma6} Set $\beta\geq\alpha$. The operator
\[%
\begin{array}
[c]{llll}%
\boldsymbol{f}(D;\alpha): & \widetilde{H}_{\beta} & \rightarrow &
\widetilde{H}_{\beta-\alpha}\\
&  &  & \\
& \phi & \mapsto & \boldsymbol{f}(D;\alpha)\phi,
\end{array}
\]
is well-defined and continuous.
\end{lemma}

\begin{proof}
By Lemma \ref{lemma5}, it is sufficient to prove that $Im(\boldsymbol{f}%
(D;\alpha))\subseteq\widetilde{H}_{\beta-\alpha}$. Set $\theta\in
Im(\boldsymbol{f}(D;\alpha))$, i.e. $\theta=\boldsymbol{f}(D;\alpha)\phi$ for
some $\phi\in\widetilde{H}_{\beta}$. Then there exists a sequence $\{\phi
_{l}\}\subseteq{\LARGE \Phi}$ such that $\phi_{l}$ $\underrightarrow
{||\cdot||_{\beta}}$ $\phi$. Define $\theta_{l}=\boldsymbol{f}(D;\alpha
)\phi_{l}\in{\LARGE \Phi}$, by the continuity of the operator $\boldsymbol{f}%
(D;\alpha)$, we have $\boldsymbol{f}(D;\alpha)\phi_{l}$ $\underrightarrow
{||\cdot||_{\beta-\alpha}}$ $\boldsymbol{f}(D;\alpha)\phi$, i.e. $\theta_{l}$
$\underrightarrow{||\cdot||_{\beta-\alpha}}\ \theta$, hence $\theta
\in\widetilde{H}_{\beta-\alpha}$.
\end{proof}

\begin{theorem}
\label{surjective} Set $\beta\geq\alpha$. The operator
\[%
\begin{array}
[c]{llll}%
\boldsymbol{f}(D;\alpha): & \widetilde{H}_{\beta} & \rightarrow &
\widetilde{H}_{\beta-\alpha}\\
&  &  & \\
& \phi & \mapsto & \boldsymbol{f}(D;\alpha)\phi,
\end{array}
\]
is a bicontinuous isomorphism of Banach spaces.
\end{theorem}

\begin{proof}
After Lemmas \ref{lemma5}-\ref{lemma6}, $\boldsymbol{f}(D;\alpha
):\widetilde{H}_{\beta}\rightarrow\widetilde{H}_{\beta-\alpha}$ is a
well-defined continuous operator. We now prove the surjectivity of
$\boldsymbol{f}(D;\alpha)$. Set $\varphi\in\widetilde{H}_{\beta-\alpha}$, then
there exists a Cauchy sequence $\{\varphi_{l}\}\subseteq{\LARGE \Phi}$
converging to $\varphi$, i.e. $\varphi_{l}$ $\underrightarrow{||\cdot
||_{\beta-\alpha}}$ $\varphi$. For each $\varphi_{l}$ we define $u_{l}$ as
follows:
\[
\mathcal{F}(u_{l})(\xi)=%
\begin{cases}
\frac{\mathcal{F}\left(  \varphi_{l}\right)  (\xi)}{|f(\xi)|_{p}^{\alpha}} &
\xi\neq0\\
0 & \xi=0.
\end{cases}
\]
Then $u_{l}\in{\LARGE \Phi}$ and $\boldsymbol{f}(D;\alpha)u_{l}=\varphi_{l}$.
The sequence $\{u_{l}\}$ is a Cauchy sequence:
\begin{align*}
||u_{l}-u_{m}||_{\beta}^{2}  &  =\int_{\mathbb{Q}_{p}^{n}}\Xi^{2\beta}%
(\xi)|\mathcal{F}\left(  u_{l}\right)  (\xi)-\mathcal{F}\left(  u_{m}\right)
(\xi)|^{2}d^{n}\xi\\
&  =\int_{\mathbb{Q}_{p}^{n}}\Xi^{2\beta}(\xi)\left\vert \frac{\mathcal{F}%
\left(  \varphi_{l}\right)  (\xi)-\mathcal{F}\left(  \varphi_{m}\right)
(\xi)}{|f(\xi)|_{p}^{\alpha}}\right\vert ^{2}d^{n}\xi\\
&  \leq C_{1}\int_{\mathbb{Q}_{p}^{n}}\Xi^{2\beta-2\alpha}(\xi)|\mathcal{F}%
\left(  \varphi_{l}\right)  (\xi)-\mathcal{F}\left(  \varphi_{m}\right)
(\xi)|^{2}d^{n}\xi\\
&  \leq C_{1}||\varphi_{l}-\varphi_{m}||_{\beta-\alpha}^{2},
\end{align*}
cf. Theorem \ref{Theo1}. Therefore $u_{l}$ $\underrightarrow{||\cdot||_{\beta
}\text{\ }}u\in\widetilde{H}_{\beta}$. By the continuity of $\boldsymbol{f}%
(D;\alpha)$, we have $\varphi_{l}=\boldsymbol{f}(D;\alpha)u_{l}$
$\underrightarrow{||\cdot||_{\beta-\alpha}}$ $\boldsymbol{f}(D;\alpha)u$ and
by the uniqueness of the limit in a metric space, we conclude $\varphi
=\boldsymbol{f}(D;\alpha)u$. Furthermore,
\begin{align*}
||\varphi||_{\beta-\alpha}  &  =||\boldsymbol{f}(D;\alpha)u||_{\beta-\alpha
}=\lim_{l\rightarrow\infty}||\boldsymbol{f}(D;\alpha)u_{l}||_{\beta-\alpha}\\
&  \leq A_{1}^{\alpha}||u||_{\beta},
\end{align*}
cf. Theorem \ref{Theo1}, which shows the continuity of $\boldsymbol{f}%
(D;\alpha)^{-1}$.

Finally we show the injectivity of $\boldsymbol{f}(D;\alpha)$. Suppose that
$\boldsymbol{f}(D;\alpha)u=0$ has a solution $u\in\widetilde{H}_{\beta}$,
$u\neq0$. We may assume that $||u||_{\beta}=1$. There exists a sequence
$\left\{  u_{l}\right\}  \subset{\LARGE \Phi}$ such that $u_{l}$
$\underrightarrow{||\cdot||_{\beta}}$ $u$, in addition, we may assume that
$||u_{l}||_{\beta}\geq\frac{1}{2}$. Set $\varphi_{l}:=\boldsymbol{f}%
(D;\alpha)u_{l}$. It follows from the continuity of $\boldsymbol{f}(D;\alpha)$
that $\varphi_{l}$ $\underrightarrow{||\cdot||_{\beta-\alpha}}$ $0$. On the
other hand, by using Theorem \ref{Theo1}, one gets%
\begin{align*}
||\varphi_{l}||_{\beta-\alpha}^{2}  &  =\int_{\mathbb{Q}_{p}^{n}}\Xi^{2\left(
\beta-\alpha\right)  }(\xi)\left\vert f\left(  \xi\right)  \right\vert
_{p}^{2\alpha}|\mathcal{F}\left(  u_{l}\right)  (\xi)|^{2}d^{n}\xi\\
&  \geq A_{0}^{2\alpha}\int_{\mathbb{Q}_{p}^{n}}|\mathcal{F}\left(
u_{l}\right)  (\xi)|^{2}d^{n}\xi=A_{0}^{2\alpha}||u_{l}||_{\beta}^{2}\geq
\frac{1}{4}A_{0}^{2\alpha},
\end{align*}
which contradicts $\varphi_{l}$ $\underrightarrow{||\cdot||_{\beta-\alpha}}$
$0$.
\end{proof}

\subsection{\label{invariant}Invariant spaces under the action of
$\boldsymbol{f}(D;\alpha)$}

Consider on $\cap_{\beta\in\mathbb{N}}\widetilde{H}_{\beta}$ the family of
seminorms $\{||\cdot||_{\beta};\ \beta\in\mathbb{N}\}$, then $\left(
\cap_{\beta\in\mathbb{N}}\widetilde{H}_{\beta},||\cdot||_{\beta};\ \beta
\in\mathbb{N}\right)  $ becomes a locally convex space, which is metrizable.
Indeed,
\begin{equation}
\rho(x,y)=\max_{\beta}\left\{  c_{\beta}\frac{||x-y||_{\beta}}%
{1+||x-y||_{\beta}}\right\}  , \label{metric}%
\end{equation}
where $\{c_{\beta}\}_{\beta\in\mathbb{N}}$ is a null-sequence of positive
numbers, is a metric for the topology of $X$, see e.g. \cite{Ru}.

Since $\{c_{\beta}\}_{\beta\in\mathbb{N}}$ is a null-sequence, there exists
$c>0$ such that $c_{\beta}\leq c$ for all $\beta\in\mathbb{N}$, therefore
\begin{equation}
\rho(x,y)=\max_{\beta}\left\{  c_{\beta}\frac{||x-y||_{\beta}}%
{1+||x-y||_{\beta}}\right\}  \leq c\max_{\beta}\{||x-y||_{\beta}\}.
\label{distancia}%
\end{equation}

\begin{lemma}
\label{lemma7} The sequence $\{x_{n}\}\subseteq\cap_{\beta\in\mathbb{N}%
}\widetilde{H}_{\beta}$ converges to $y\in\cap_{\beta\in\mathbb{N}}%
\widetilde{H}_{\beta}$ in the metric $\rho$, if and only if $\{x_{n}\}$
converges to $y$ in the norm $||\cdot||_{\beta}$ for all $\beta\in\mathbb{N}$.
\end{lemma}

\begin{proposition}
\label{espacio} (1) Set $\overline{({\LARGE \Phi},\rho)}$ for the completion
of the metric space $({\LARGE \Phi},\rho)$, and $\widetilde{H}_{\infty
}:=\overline{(\cap_{\beta\in\mathbb{N}}\widetilde{H}_{\beta},\rho)}$ for the
completion of the metric space $(\cap_{\beta\in\mathbb{N}}\widetilde{H}%
_{\beta},\rho)$. Then
\[
\overline{({\LARGE \Phi},\rho)}=\widetilde{H}_{\infty},
\]
as complete metric spaces.

(2) $\cap_{\beta\in\mathbb{N}}\widetilde{H}_{\beta}\neq\emptyset$ and
$\widetilde{H}_{\infty}=(\cap_{\beta\in\mathbb{N}}\widetilde{H}_{\beta},\rho)$.
\end{proposition}

\begin{proof}
(1) Set $\phi\in\overline{({\LARGE \Phi},\rho)}$, then there exists a sequence
$\{\phi_{l}\}\subseteq{\LARGE \Phi}$ such that $\phi_{l}$ $\underrightarrow
{\rho\text{\ }}\phi$. By the previous lemma $\phi_{l}$ $\underrightarrow
{||\cdot||_{\beta}\text{\ }}\phi$ for each $\beta\in\mathbb{N}$, i.e. $\phi
\in\cap_{\beta\in\mathbb{N}}\widetilde{H}_{\beta}$. Therefore, $\overline
{({\LARGE \Phi},\rho)}\subseteq(\cap_{\beta\in\mathbb{N}}\widetilde{H}_{\beta
},\rho)\subseteq\overline{(\cap_{\beta\in\mathbb{N}}\widetilde{H}_{\beta}%
,\rho)}$.

Conversely, set $\varphi\in\widetilde{H}_{\infty}$, then there exists a
sequence $\{\varphi_{l}\}\subseteq\cap_{\beta\in\mathbb{N}}\widetilde
{H}_{\beta}$ sa\-tis\-fying $\varphi_{l}$ $\underrightarrow{\rho\text{\ }%
}\varphi$. Then, for each $\varphi_{l}$ and for each $\beta\in\mathbb{N}$,
there exists an element $\varphi_{\beta,m(l)}\in{\LARGE \Phi}$ such that
$||\varphi_{\beta,m(l)}-\varphi_{l}||_{\beta}<\frac{1}{l+1}$, for all
$\beta\in\mathbb{N}$. Then
\begin{align*}
\rho(\varphi_{\beta,m(l)},\varphi)  &  \leq\rho(\varphi_{\beta,m(l)}%
,\varphi_{l})+\rho(\varphi_{l},\varphi)\\
&  \leq c\max_{\beta}||\varphi_{\beta,m(l)}-\varphi_{l}||_{\beta}+\rho
(\varphi_{l},\varphi)\\
&  \leq c\frac{1}{l+1}+\rho(\varphi_{l},\varphi)\rightarrow0.
\end{align*}
where we have used inequality (\ref{distancia}). This shows that the sequence
$\{\varphi_{\beta,m(l)}\}\subseteq{\LARGE \Phi}$ satisfies $\varphi
_{\beta,m(l)}$ $\underrightarrow{\rho\text{\ }}\varphi$, and thus $\varphi
\in\overline{({\LARGE \Phi},\rho)}$.

(2) It follows from the first part.
\end{proof}

\begin{proposition}
\label{Prop3} The operator
\[%
\begin{array}
[c]{llll}%
\boldsymbol{f}(D;\alpha): & \widetilde{H}_{\infty} & \rightarrow &
\widetilde{H}_{\infty}\\
&  &  & \\
& \phi & \mapsto & \boldsymbol{f}(D;\alpha)\phi,
\end{array}
\]
is well-defined and continuous.
\end{proposition}

\begin{proof}
Set $\phi\in\widetilde{H}_{\infty}=(\cap_{\gamma\in\mathbb{N}}\widetilde
{H}_{\gamma},\rho)$. Take $\gamma=\beta+\alpha$, with $\beta\geq0$. By using
Theorem \ref{surjective} we find that $\boldsymbol{f}(D;\alpha)\phi
\in\widetilde{H}_{\beta}$ for all $\beta\geq0$, thus, $\boldsymbol{f}%
(D;\alpha)\phi\in(\cap_{\beta\in\mathbb{N}}\widetilde{H}_{\beta},\rho)$.

In order to prove the continuity of the operator, let $\{\phi_{l}%
\}\subseteq\cap_{\gamma\in\mathbb{N}}\widetilde{H}_{\gamma}$ be a sequence
such that $\rho(\phi_{l},\phi)\rightarrow0$ as $l\rightarrow\infty$, with
$\phi\in\cap_{\gamma\in\mathbb{N}}\widetilde{H}_{\gamma}$. By Lemma
\ref{lemma7} $||\phi_{l}-\phi||_{\gamma}\rightarrow0$ as $l\rightarrow\infty$
for all $\gamma\in\mathbb{N}$. Take $\gamma=\beta+\alpha$, with $\beta\geq0$.
Then $||\boldsymbol{f}(D;\alpha)\phi_{l}-\boldsymbol{f}(D;\alpha)\phi
||_{\beta}\rightarrow0$ (cf. Theorem \ref{surjective}). Therefore,
$\rho(\boldsymbol{f}(D;\alpha)\phi_{l},\boldsymbol{f}(D;\alpha)\phi
)\rightarrow0$ as $l\rightarrow\infty$ (cf. Lemma \ref{lemma7}), and thus the
operator defined over $(\cap_{\beta\in\mathbb{N}}\widetilde{H}_{\beta},\rho)$
is continuous.
\end{proof}

\begin{theorem}
\label{Theo6} The operator
\[%
\begin{array}
[c]{llll}%
\boldsymbol{f}(D;\alpha): & \widetilde{H}_{\infty} & \rightarrow &
\widetilde{H}_{\infty}\\
&  &  & \\
& \phi & \mapsto & \boldsymbol{f}(D;\alpha)\phi,
\end{array}
\]
is a bicontinuous isomorphism of locally convex vector spaces.
\end{theorem}

\begin{proof}
After Proposition \ref{Prop3}, the operator $\boldsymbol{f}(D;\alpha
):\widetilde{H}_{\infty}\rightarrow\widetilde{H}_{\infty}$ is well-defined and
continuous. We now prove the surjectivity of $\boldsymbol{f}(D;\alpha)$. Set
$\phi\in\widetilde{H}_{\infty}$, then there exists a Cauchy sequence
$\{\phi_{l}\}\subseteq{\LARGE \Phi}$ satisfying $\phi_{l}$ $\underrightarrow
{\rho\text{\ }}\phi.$ For each $\phi_{l}$ we can define $u_{l}$ as in the
proof of Theorem \ref{surjective}:
\[
\mathcal{F}(u_{l})(\xi)=%
\begin{cases}
\frac{\mathcal{F}(\phi_{l})(\xi)}{|f(\xi)|_{p}^{\alpha}}, & \xi\neq0\\
0 & \xi=0.
\end{cases}
\]
Then $\{u_{l}\}\subseteq{\LARGE \Phi}$ is a Cauchy sequence in each norm
$||\cdot||_{\beta}$ and therefore is a Cauchy sequence in the metric $\rho$.
Therefore there exists $u\in\widetilde{H}_{\infty}$ such that $u_{l}$
$\underrightarrow{\rho\text{\ }}u$, then by using the continuity of the
operator, cf. Proposition \ref{Prop3}, we have $\phi_{l}=\boldsymbol{f}%
(D;\alpha)u_{l}$ $\underrightarrow{\rho\text{\ }}\boldsymbol{f}(D;\alpha)u$,
i.e. $\boldsymbol{f}(D;\alpha)u=\phi$. The continuity of $\boldsymbol{f}%
(D;\alpha)^{-1}$ is established as in the proof of Theorem \ref{surjective}.

Finally, we show the injectivity of $\boldsymbol{f}(D;\alpha)$. Suppose that
$\boldsymbol{f}(D;\alpha)u=0$ has a solution $u\in\widetilde{H}_{\infty}$.
Since $\widetilde{H}_{\infty}=\cap_{\beta\in\mathbb{N}}\widetilde{H}_{\beta}$,
cf. Proposition \ref{espacio}, $\boldsymbol{f}(D;\alpha)u=0$ in all
$\widetilde{H}_{\beta}$, by the proof of Theorem \ref{surjective}, we know
that it is possible only for $u=0$.
\end{proof}

\begin{remark}
\label{Taibleson}The Taibleson operator $D_{T}^{\alpha}$, with $\alpha>0$, is
defined as $D_{T}^{\alpha}\phi\left(  x\right)  =\mathcal{F}_{\xi\rightarrow
x}^{-1}\left(  \left\Vert \xi\right\Vert _{p}^{\alpha}\mathcal{F}%
_{x\rightarrow\xi}\left(  \phi\right)  \right)  $ for $\phi\in S$. The
operator
\[%
\begin{array}
[c]{llll}%
D_{T}^{\alpha}: & \widetilde{H}_{\infty} & \rightarrow & \widetilde{H}%
_{\infty}\\
&  &  & \\
& \phi & \mapsto & D_{T}^{\alpha}\phi,
\end{array}
\]
is a bicontinuous isomorphism of locally convex vector spaces. The proof of
this results is similar to the one given for Theorem \ref{Theo6}.
\end{remark}

\section{Infinitely Pseudo-differentiable Functions}

\subsection{\label{Sect_family}A Family of infinitely pseudo-differentiable
functions}

Set
\[
\Omega_{l}\left(  x\right)  :=\left\{
\begin{array}
[c]{ccc}%
1 & \text{if} & \left\Vert x\right\Vert _{p}\leq p^{l}\\
&  & \\
0 & \text{if} & \left\Vert x\right\Vert _{p}>p^{l},
\end{array}
\right.
\]
for $l\in\mathbb{Z}$, and $\varphi_{r,t}\left(  x\right)  :=\mathcal{F}%
^{-1}\left(  \left(  1-\Omega_{-r}\left(  \xi\right)  \right)  e^{-t|f(\xi
)|_{p}^{\alpha}}\right)  $ for $r\in\mathbb{N\smallsetminus}\left\{
0\right\}  $, $t$, $\alpha>0$, and $\varphi_{l,r,t}\left(  x\right)
:=\mathcal{F}^{-1}\left(  \left(  1-\Omega_{-r}\left(  \xi\right)  \right)
\Omega_{l}\left(  \xi\right)  e^{-t|f(\xi)|_{p}^{\alpha}}\right)  $ for
$r\in\mathbb{N\smallsetminus}\left\{  0\right\}  $, $l\in\mathbb{N}$, with
$l>r$, $\ $and $Z\left(  x,t,\alpha\right)  :=Z\left(  x,t\right)
=\int_{\mathbb{Q}_{p}^{n}}\Psi\left(  x\cdot\xi\right)  e^{-t\left\vert
f\left(  \xi\right)  \right\vert _{p}^{\alpha}}d^{n}\xi$.

\begin{remark}
\label{note3}Let $\left\{  \theta_{l}\right\}  \subset L^{2}$and $\theta\in
L^{2}$. We recall that if
\begin{equation}
\lim_{_{l\rightarrow\infty}}%
{\textstyle\int\limits_{\mathbb{Q}_{p}^{n}}}
|\mathcal{F}(\theta_{l})(\xi)-\mathcal{F}(\theta)(\xi)|^{2}d^{n}\xi
=\lim_{_{l\rightarrow\infty}}%
{\textstyle\int\limits_{\mathbb{Q}_{p}^{n}}}
|\theta_{l}(\xi)-\theta(\xi)|^{2}d^{n}\xi=0, \label{eq8}%
\end{equation}
then by the Chebishev inequality%
\[
\lim_{l\rightarrow\infty}vol\left(  \left\{  \xi\in\mathbb{Q}_{p}^{n}%
;|\theta_{l}(\xi)-\theta(\xi)|\geq\delta\right\}  \right)  =0,
\]
for any $\delta\geq0$ (here $vol(A)$ means the Haar measure of $A$), i.e.
$\theta_{l}$ $\underrightarrow{\text{measure}}$ $\theta$ as $l\rightarrow
\infty$. We also recall that if
\[
\varphi_{l}\text{ }\underrightarrow{\text{measure}}\text{ }\widetilde{\varphi
}_{0}\text{ and }\varphi_{l}\text{ }\underrightarrow{\text{measure}}\text{
}\widetilde{\varphi}_{1},
\]
then $\widetilde{\varphi}_{0}=\widetilde{\varphi}_{1}$ almost everywhere.
\end{remark}

\begin{lemma}
\label{note4}With the above notation, $\varphi_{l,r,t}\left(  x\right)
\in{\LARGE \Phi}$.
\end{lemma}

\begin{proof}
Note that the support of $\mathcal{F}\left(  \varphi_{l,r,t}\right)  \left(
\xi\right)  $\ is $A:=\left\{  \xi\in\mathbb{Q}_{p}^{n};p^{-r+1}\leq\left\Vert
\xi\right\Vert _{p}\leq p^{l}\right\}  $, which is a compact subset. By
applying Lemma \ref{lemma3} (2) to $f$ and $A$, we get a finite number of
points $\widetilde{\xi}_{i}\in A$, such that
\[
\varphi_{l,r,t}\left(  x\right)  =\mathcal{F}_{\xi\rightarrow x}^{-1}\left(
\sum\limits_{i}e^{-t\left\vert f\left(  \widetilde{\xi}_{i}\right)
\right\vert _{p}^{\alpha}}\Omega_{M+R+1}\left(  \xi-\widetilde{\xi}%
_{i}\right)  \right)  ,
\]
which shows that $\varphi_{l,r,t}\left(  x\right)  \in{\LARGE \Phi}$.
\end{proof}

\begin{lemma}
\label{lemma10}With the above notation, the following assertions hold: (1)
$\left\{  \varphi_{l,r,t}\right\}  _{l}$ is Cauchy sequence in $\widetilde
{H}_{\infty}$, (2) $\varphi_{l,r,t}$ $\underrightarrow{\text{measure}}$
$\varphi_{r,t}$ as $l\rightarrow\infty$.
\end{lemma}

\begin{proof}
(1) Take $m\geq l>r$, then $\left\Vert \varphi_{m,r,t}-\varphi_{l,r,t}%
\right\Vert _{\beta}^{2}$ equals%
\[
\int\limits_{\mathbb{Q}_{p}^{n}}\Xi^{2\beta}(\xi)|\left(  1-\Omega_{-r}\left(
\xi\right)  \right)  \Omega_{m}\left(  \xi\right)  e^{-t\left\vert f\left(
\xi\right)  \right\vert _{p}^{\alpha}}-\left(  1-\Omega_{-r}\left(
\xi\right)  \right)  \Omega_{l}\left(  \xi\right)  e^{-t\left\vert f\left(
\xi\right)  \right\vert _{p}^{\alpha}}|^{2}d^{n}\xi
\]%
\begin{align*}
&  \leq\int\limits_{p^{l+1}\leq\left\Vert \xi\right\Vert _{p}\leq p^{m}}%
\Xi^{2\beta}(\xi)e^{-2t\left\vert f\left(  \xi\right)  \right\vert
_{p}^{\alpha}}d^{n}\xi\\
&  \leq\int\limits_{\left\Vert \xi\right\Vert _{p}\geq p^{l+1}}\Xi^{2\beta
}(\xi)e^{-2A_{0}^{\alpha}t\left\Vert \xi\right\Vert _{p}^{\alpha d}}d^{n}\xi\\
&  \leq n^{2\beta}\int\limits_{\left\Vert \xi\right\Vert _{p}\geq p^{l+1}%
}\left\Vert \xi\right\Vert _{p}^{\frac{2\beta d}{\min_{i}\left\{
w_{i}\right\}  }}e^{-2A_{0}^{\alpha}t\left\Vert \xi\right\Vert _{p}^{\alpha
d}}d^{n}\xi,
\end{align*}
where in the last inequality we use that $\Xi(\xi)\leq n\left\Vert
\xi\right\Vert _{p}^{\frac{d}{\min_{i}\left\{  w_{i}\right\}  }}$ for
$\left\Vert \xi\right\Vert _{p}>1$ and Theorem \ref{Theo1}. Now, since
$\int\nolimits_{\mathbb{Q}_{p}^{n}}\left\Vert \xi\right\Vert _{p}%
^{\frac{2\beta d}{\min_{i}\left\{  w_{i}\right\}  }}e^{-2A_{0}^{\alpha
}t\left\Vert \xi\right\Vert _{p}^{\alpha d}}d^{n}\xi<\infty$, by using the
Lebesgue dominated convergence lemma, one gets
\[
\left\Vert \varphi_{m,r,t}-\varphi_{l,r,t}\right\Vert _{\beta}^{2}%
\rightarrow0\text{ as }m,l\rightarrow\infty.
\]

(2) By Remark (\ref{note3}), it is sufficient to show:%

\begin{align*}
&  \lim_{_{l\rightarrow\infty}}\int\limits_{\mathbb{Q}_{p}^{n}}|\left(
1-\Omega_{-r}\left(  \xi\right)  \right)  \Omega_{l}\left(  \xi\right)
e^{-t\left\vert f\left(  \xi\right)  \right\vert _{p}^{\alpha}}-\left(
1-\Omega_{-r}\left(  \xi\right)  \right)  e^{-t\left\vert f\left(  \xi\right)
\right\vert _{p}^{\alpha}}|^{2}d^{n}\xi\\
&  \leq\lim_{_{l\rightarrow\infty}}\int\limits_{\left\{  \xi;\left\Vert
\xi\right\Vert _{p}\geq p^{l+1}\right\}  }e^{-2A_{0}^{\alpha}t\left\Vert
\xi\right\Vert _{p}^{\alpha d}}d^{n}\xi=0,
\end{align*}
by Lebesgue dominated convergence lemma and $\int_{\mathbb{Q}_{p}^{n}%
}e^{-2A_{0}^{\alpha}t\left\Vert \xi\right\Vert _{p}^{\alpha d}}d^{n}\xi
<\infty$.
\end{proof}

\begin{theorem}
\label{Theo7A}(1) $\varphi_{r,t}\left(  x\right)  \in\widetilde{H}_{\infty}$,
for any $r\in\mathbb{N\smallsetminus}\left\{  0\right\}  $, $t>0$. (2)
$Z\left(  x,t\right)  \in\widetilde{H}_{\infty}$, for any $t>0$.
\end{theorem}

\begin{proof}
(1) By Lemma \ref{lemma10} (1), $\left\{  \varphi_{l,r,t}\right\}  _{l}$ is
Cauchy sequence in $\widetilde{H}_{\infty}$, i.e. $\varphi_{l,r,t}$
$\underrightarrow{||\cdot||_{\beta}\text{\ }}\widetilde{\varphi}$, for some
$\widetilde{\varphi}$ $\in\widetilde{H}_{\infty}$ and any $\beta\in\mathbb{N}%
$. For $\beta=0$, this means $\varphi_{l,r,t}$ $\underrightarrow
{L^{2}\text{\ }}\widetilde{\varphi}\in L^{2}$, and by Remark \ref{note3},
\[
\varphi_{l,r,t}\text{ }\underrightarrow{\text{measure}}\text{ }\widetilde
{\varphi}\text{ as }l\rightarrow\infty.
\]
On the other hand, by Lemma \ref{lemma10} (2),
\[
\varphi_{l,r,t}\text{ }\underrightarrow{\text{measure}}\text{ }\varphi
_{r,t}\text{ as }l\rightarrow\infty,
\]
and then by Remark \ref{note3}, $\widetilde{\varphi}=\varphi_{r,t}$ almost
everywhere, which implies that $\varphi_{r,t}\in\widetilde{H}_{\infty}$.

(2) We first note that $\left\{  \varphi_{r,t}\right\}  _{r}$ is a \ Cauchy
sequence in $\left\Vert \cdot\right\Vert _{\beta}$, for any $\beta
\in\mathbb{N}$. Indeed, \ set $r\geq s$, since
\[
\varphi_{l,r,t}\text{ }\underrightarrow{\left\Vert \cdot\right\Vert _{\beta}%
}\text{ }\varphi_{r,t}\text{ as }l\rightarrow\infty,\text{ and }%
\varphi_{m,s,t}\text{ }\underrightarrow{\left\Vert \cdot\right\Vert _{\beta}%
}\text{ }\varphi_{s,t}\text{ as }m\rightarrow\infty,
\]
by the first part, we have%
\begin{align*}
\left\Vert \varphi_{r,t}-\varphi_{s,t}\right\Vert _{\beta}^{2}  &
=\lim_{l\rightarrow\infty}\lim_{m\rightarrow\infty}\left\Vert \varphi
_{l,r,t}\text{ }-\varphi_{m,s,t}\right\Vert _{\beta}^{2}\\
&  =\lim_{l\rightarrow\infty}\lim_{m\rightarrow\infty}\int\limits_{\mathbb{Q}%
_{p}^{n}}\Xi^{2\beta}(\xi)\left\vert \mathcal{F}\left(  \varphi_{l,r,t}%
\right)  (\xi)-\mathcal{F}\left(  \varphi_{m,s,t}\right)  (\xi)\right\vert
^{2}d^{n}\xi.
\end{align*}

We now note that
\[
\Xi^{2\beta}(\xi)\left\vert \mathcal{F}\left(  \varphi_{l,r,t}\right)  \text{
}(\xi)-\mathcal{F}\left(  \varphi_{m,s,t}\right)  (\xi)\right\vert ^{2}%
\leq4\Xi^{2\beta}(\xi)e^{-2t\left\vert f\left(  \xi\right)  \right\vert
_{p}^{\alpha}}\in L^{1},
\]
for any $\alpha,\beta,t>0$. By using the Lebesgue dominated convergence lemma
and Theorem \ref{Theo1},%

\begin{align*}
\left\Vert \varphi_{r,t}\text{ }-\varphi_{s,t}\right\Vert _{\beta}^{2}  &
=\int\limits_{\mathbb{Q}_{p}^{n}}\Xi^{2\beta}(\xi)\left\vert \mathcal{F}%
\left(  \varphi_{r,t}\right)  \text{ }(\xi)-\mathcal{F}\left(  \varphi
_{s,t}\right)  (\xi)\right\vert ^{2}d^{n}\xi\\
&  =\int\limits_{p^{-s}\mathbb{<}\left\Vert \xi\right\Vert _{p}\leq p^{-r}}%
\Xi^{2\beta}(\xi)e^{-2t\left\vert f\left(  \xi\right)  \right\vert
_{p}^{\alpha}}d^{n}\xi\\
&  \leq\int\limits_{\left\Vert \xi\right\Vert _{p}\leq p^{-r}}\Xi^{2\beta}%
(\xi)e^{-2t\left\vert f\left(  \xi\right)  \right\vert _{p}^{\alpha}}d^{n}%
\xi\rightarrow0\text{ as }r\rightarrow\infty.
\end{align*}

Therefore $\varphi_{r,t}$ $\underrightarrow{\left\Vert \cdot\right\Vert
_{\beta}}$ $\widetilde{\varphi}$, for some $\widetilde{\varphi}\in
\widetilde{H}_{\infty}$, for any $\beta\in\mathbb{N}$. By taking $\beta=0$, we
get $\varphi_{r,t}$ $\underrightarrow{\text{measure}}$ $\widetilde{\varphi}$.

On the other hand, by reasoning as in the proof of Lemma \ref{lemma10}, one
shows that $\varphi_{r,t}$ $\underrightarrow{\text{measure}}$ $Z(x,t)$ as
$r\rightarrow\infty$, and then $Z(x,t)\in\widetilde{H}_{\infty}$.
\end{proof}

\begin{remark}
\label{note5}By using the technique above presented one shows that
\[
\mathcal{F}^{-1}\left(  \left(  1-\Omega_{-r}\left(  \xi\right)  \right)
e^{-t\left\Vert \xi\right\Vert _{p}^{\alpha}}\right)  ,\mathcal{F}^{-1}\left(
e^{-t\left\Vert \xi\right\Vert _{p}^{\alpha}}\right)  \in\widetilde{H}%
_{\infty},
\]
for $r\in\mathbb{N}$, $t>0$, $\alpha>0$.
\end{remark}

\subsection{\label{Sect_reg}Regularization Effect in Parabolic Equations}

\begin{lemma}
\label{lemma8}Let $\varphi\in S$ and $\theta\in\widetilde{H}_{\infty}\cap
L^{1}$. Then $\varphi\ast\theta\in\widetilde{H}_{\infty}$.
\end{lemma}

\begin{proof}
There exists a sequence $\left\{  \theta_{l}\right\}  \subset{\LARGE \Phi}$
such that $\theta_{l}$ $\underrightarrow{||\cdot||_{\beta}\text{\ }}\theta$
for any $\beta\in\mathbb{N}$. We claim that $\left\{  \varphi\ast\theta
_{l}\right\}  $ is a Cauchy sequence in $||\cdot||_{\beta}$, for any $\beta
\in\mathbb{N}$. Indeed, let $m\geq l$, then%

\begin{align*}
\left\Vert \varphi\ast\theta_{m}-\varphi\ast\theta_{l}\right\Vert _{\beta
}^{2}  &  =\int\limits_{\mathbb{Q}_{p}^{n}}\Xi^{2\beta}(\xi)|\mathcal{F}%
\left(  \varphi\right)  \left(  \xi\right)  |^{2}\left\vert \mathcal{F}\left(
\theta_{m}\right)  \left(  \xi\right)  -\mathcal{F}\left(  \theta_{l}\right)
\left(  \xi\right)  \right\vert ^{2}d^{n}\xi\\
&  \leq C\left(  \varphi\right)  \left\Vert \theta_{m}-\theta_{l}\right\Vert
_{\beta}^{2}\rightarrow0\text{, as }m,l\rightarrow\infty\text{.}%
\end{align*}

Thus, there exists $\widetilde{\theta}\in\widetilde{H}_{\infty}$ such that
$\varphi\ast\theta_{l}$ $\underrightarrow{\rho\text{\ }}\widetilde{\theta}$.
By taking $\beta=0$ this means $\varphi\ast\theta_{l}$ $\underrightarrow
{L^{2}\text{\ }}\widetilde{\theta}\in L^{2}$, and by Remark \ref{note3},
\[
\varphi\ast\theta_{l}\text{ }\underrightarrow{\text{measure}}\text{
}\widetilde{\theta}\text{ as }l\rightarrow\infty.
\]

On the other hand,%
\begin{align*}
&  \lim_{l\rightarrow\infty}\int\limits_{\mathbb{Q}_{p}^{n}}\left\vert
\varphi\ast\theta_{l}\left(  \xi\right)  -\varphi\ast\theta\left(  \xi\right)
\right\vert ^{2}d^{n}\xi\\
&  =\lim_{l\rightarrow\infty}\int\limits_{\mathbb{Q}_{p}^{n}}|\mathcal{F}%
\left(  \varphi\right)  \left(  \xi\right)  |^{2}\left\vert \mathcal{F}\left(
\theta_{l}\right)  \left(  \xi\right)  -\mathcal{F}\left(  \theta\right)
\left(  \xi\right)  \right\vert ^{2}d^{n}\xi\\
&  \leq C\left(  \varphi\right)  \lim_{l\rightarrow\infty}\int
\limits_{\mathbb{Q}_{p}^{n}}\left\vert \mathcal{F}\left(  \theta_{l}\right)
\left(  \xi\right)  -\mathcal{F}\left(  \theta\right)  \left(  \xi\right)
\right\vert ^{2}d^{n}\xi\\
&  \leq C\left(  \varphi\right)  \lim_{l\rightarrow\infty}\left\Vert
\theta_{l}-\theta\right\Vert _{L^{2}}^{2}=0,
\end{align*}
since $\theta_{l}$ $\underrightarrow{||\cdot||_{\beta}\text{\ }}\theta$ for
any $\beta\in\mathbb{N}$. Now by Remark \ref{note3}, we have
\[
\varphi\ast\theta_{l}\text{ }\underrightarrow{\text{measure}}\text{ }%
\varphi\ast\theta\text{ as }l\rightarrow\infty.
\]
Therefore $\widetilde{\theta}=\varphi\ast\theta$, almost everywhere, which
implies that $\varphi\ast\theta\in\widetilde{H}_{\infty}$.
\end{proof}

\begin{lemma}
\label{lemma8A}(1) Let $\theta\in\widetilde{H}_{\beta}\cap L^{1}$ for any
$\beta\in$ $\mathbb{N}$. Set%
\[%
\begin{array}
[c]{cccc}%
T_{\theta}: & \left(  S,\left\Vert \cdot\right\Vert _{L^{2}}\right)  &
\rightarrow & \left(  \widetilde{H}_{\beta},\left\Vert \cdot\right\Vert
_{\beta}\right) \\
&  &  & \\
& \varphi & \rightarrow & \varphi\ast\theta,
\end{array}
\]
for $\beta\in$ $\mathbb{N}$. Then $\varphi\ast\theta\in\widetilde{H}_{\infty}%
$. Then $T_{\theta}$\ is a linear continuous mapping for any $\beta\in$
$\mathbb{N}$.

(2) The mapping $T_{\theta}$ has unique continuous extension from $L^{2}$ into
$\widetilde{H}_{\infty}$.
\end{lemma}

\begin{proof}
Since $\widetilde{H}_{\infty}=\cap_{\beta\in\mathbb{N}}\widetilde{H}_{\beta}$,
cf. Proposition \ref{espacio} (2), and by using Lemma \ref{lemma8}, we have
$T_{\theta}$ is well-defined. Let $\left\{  \theta_{l}\right\}  \subset
{\LARGE \Phi}$ be a sequence such that%
\begin{equation}
\theta_{l}\underrightarrow{||\cdot||_{\beta}\text{\ }}\theta\text{ for any
}\beta\in\mathbb{N} \label{seq}%
\end{equation}
as in the proof of Lemma \ref{lemma8}. Then
\begin{align*}
||\varphi\ast\theta||_{\beta}^{2}  &  =\lim_{l\rightarrow\infty}||\varphi
\ast\theta_{l}||_{\beta}^{2}\\
&  =\lim_{l\rightarrow\infty}\int\limits_{\mathbb{Q}_{p}^{n}}\Xi^{2\beta
}\left(  \xi\right)  \left\vert \mathcal{F}\left(  \varphi\right)  \left(
\xi\right)  \right\vert ^{2}\left\vert \mathcal{F}\left(  \theta_{l}\right)
\left(  \xi\right)  \right\vert ^{2}d^{n}\xi.
\end{align*}
On the other hand, taking $\beta=0$ in (\ref{seq}), one gets $\mathcal{F}%
\left(  \theta_{l}\right)  \underrightarrow{\text{ }L^{2}\text{\ }}%
\mathcal{F}\left(  \theta\right)  $, and by Remark \ref{note3},%
\[
\lim_{l\rightarrow\infty}vol\left(  \left\{  \xi\in\mathbb{Q}_{p}%
^{n};\left\vert \mathcal{F}\left(  \theta_{l}\right)  \left(  \xi\right)
-\mathcal{F}\left(  \theta\right)  \left(  \xi\right)  \right\vert
>\delta\right\}  \right)  =0,
\]
for any $\delta>0$. Fix $\epsilon$, $\delta>0$, then there exists $l_{0}%
\in\mathbb{N}$ such that
\begin{equation}
vol\left(  V_{\delta,l}\right)  :=vol\left(  \left\{  \xi\in\mathbb{Q}_{p}%
^{n};\left\vert \mathcal{F}\left(  \theta_{l}\right)  \left(  \xi\right)
-\mathcal{F}\left(  \theta\right)  \left(  \xi\right)  \right\vert
>\delta\right\}  \right)  <\epsilon\text{, for }l\geq l_{0}. \label{vol}%
\end{equation}

Now,
\begin{align*}
||\varphi\ast\theta||_{\beta}^{2}  &  =\lim_{l\rightarrow\infty}\left(
\int\limits_{V_{\delta,l}}\Xi^{2\beta}\left(  \xi\right)  \left\vert
\mathcal{F}\left(  \varphi\right)  \left(  \xi\right)  \right\vert
^{2}\left\vert \mathcal{F}\left(  \theta_{l}\right)  \left(  \xi\right)
\right\vert ^{2}d^{n}\xi\text{ }+\right. \\
&  \left.  \int\limits_{\mathbb{Q}_{p}^{n}\smallsetminus V_{\delta,l}}%
\Xi^{2\beta}\left(  \xi\right)  \left\vert \mathcal{F}\left(  \varphi\right)
\left(  \xi\right)  \right\vert ^{2}\left\vert \mathcal{F}\left(  \theta
_{l}\right)  \left(  \xi\right)  \right\vert ^{2}d^{n}\xi\right) \\
&  :=\lim_{l\rightarrow\infty}\left(  I_{1}+I_{2}\right)  .
\end{align*}

We now consider $\lim_{l\rightarrow\infty}I_{1}$. Since $\Xi^{2\beta}\left(
\xi\right)  \left\vert \mathcal{F}\left(  \varphi\right)  \left(  \xi\right)
\right\vert ^{2}\left\vert \mathcal{F}\left(  \theta_{l}\right)  \left(
\xi\right)  \right\vert ^{2}$ is a continuous function with compact support,%
\[
\lim_{l\rightarrow\infty}I_{1}\leq\left(  \sup_{\xi\in\text{supp}\left(
\mathcal{F}\left(  \varphi\right)  \right)  }\Xi^{2\beta}\left(  \xi\right)
\left\vert \mathcal{F}\left(  \varphi\right)  \left(  \xi\right)  \right\vert
^{2}\left\vert \mathcal{F}\left(  \theta_{l}\right)  \left(  \xi\right)
\right\vert ^{2}\right)  \lim_{l\rightarrow\infty}vol\left(  V_{\delta
,l}\right)  =0.
\]

We now study $\lim_{l\rightarrow\infty}I_{2}$. Since $\theta\in L^{1}$, then
$\mathcal{F}\left(  \theta\right)  \left(  \xi\right)  $ uniformly continuous
and $\left\vert \mathcal{F}\left(  \theta\right)  \left(  \xi\right)
\right\vert \rightarrow0$ as $\left\Vert \xi\right\Vert _{p}\rightarrow\infty$
(Riemann-Lebesgue theorem), from these facts follow that $|\mathcal{F}\left(
\theta\right)  \left(  \xi\right)  |^{2}\leq C\left(  \theta\right)  $. Then%
\begin{align*}
\lim_{l\rightarrow\infty}I_{2}  &  \leq\int\limits_{\mathbb{Q}_{p}%
^{n}\smallsetminus V_{\delta,l}}\Xi^{2\beta}\left(  \xi\right)  \left\vert
\mathcal{F}\left(  \varphi\right)  \left(  \xi\right)  \right\vert ^{2}\left(
\left\vert \mathcal{F}\left(  \theta_{l}\right)  \left(  \xi\right)
-\mathcal{F}\left(  \theta\right)  \left(  \xi\right)  \right\vert +\left\vert
\mathcal{F}\left(  \theta\right)  \left(  \xi\right)  \right\vert \right)
^{2}d^{n}\xi\\
&  \leq\left(  \delta+C\left(  \theta\right)  \right)  ^{2}\int
\limits_{\mathbb{Q}_{p}^{n}}\Xi^{2\beta}\left(  \xi\right)  \left\vert
\mathcal{F}\left(  \varphi\right)  \left(  \xi\right)  \right\vert ^{2}%
d^{n}\xi=\left(  \delta+C\left(  \theta\right)  \right)  ^{2}\left\Vert
\varphi\right\Vert _{\beta}^{2}\text{.}%
\end{align*}

In conclusion,
\[
||\varphi\ast\theta||_{\beta}\leq\left(  \delta+C\left(  \theta\right)
\right)  \left\Vert \varphi\right\Vert _{\beta}\text{, for any }\beta
\in\mathbb{N}\text{.}%
\]

(2) By the first part, using the fact that $S$ is dense in $L^{2}$ and that
$\widetilde{H}_{\infty}=\cap_{\beta\in\mathbb{N}}\widetilde{H}_{\beta}$, we
obtain a unique continuous extension $T_{\theta}:L^{2}\rightarrow\widetilde
{H}_{\infty}$.
\end{proof}

\begin{example}
The mapping $T:L^{2}\rightarrow\widetilde{H}_{\infty}$, $\phi\rightarrow
\phi\ast Z(\cdot,t)$, for any $t>0$, is continuous, since $Z(\cdot,t)\in
L^{1}$, for any $t>0$, see \cite[Corollary 1]{ZG3}.

Let $C_{b}\left(  \mathbb{Q}_{p}^{n},\mathbb{C}\right)  $ denote the space of
continuous and bounded functions from $\mathbb{Q}_{p}^{n}$\ to $\mathbb{C}$.
\end{example}

\begin{theorem}
\label{Theo7} Let $f\left(  D,\alpha\right)  $ be an elliptic
pseudo-differential operator. If $L^{2}\left(  \mathbb{Q}_{p}^{n}\right)  \cap
C_{b}\left(  \mathbb{Q}_{p}^{n},\mathbb{C}\right)  $, then the Cauchy problem
\begin{equation}
\left\{
\begin{array}
[c]{l}%
\frac{\partial u\left(  x,t\right)  }{\partial t}+\left(  f\left(
D,\alpha\right)  u\right)  \left(  x,t\right)  =0,\text{ }x\in\mathbb{Q}%
_{p}^{n},\text{ }t>0,\\
\\
u\left(  x,0\right)  =\varphi\left(  x\right)  ,
\end{array}
\right.  \label{Cauchy_Problem}%
\end{equation}
has a \ classical solution:%
\[
u\left(  x,t\right)  =\int\limits_{\mathbb{Q}_{p}^{n}}Z\left(  \eta,t\right)
\varphi\left(  x-\eta\right)  d^{n}\eta,
\]
which belongs to $\widetilde{H}_{\infty}$, for every fixed $t>0$.
\end{theorem}

\begin{proof}
The fact that the initial value problem (\ref{Cauchy_Problem}) has a classical
solution follows from Lemma 5 in \cite{ZG3}, which requires an estimation for
$Z\left(  x,t\right)  $ showing its decay in space and time, such estimation
is only known in the case in which $f$ is quasielliptic of degree $d$ with
respect to $(1,\ldots,1)$, see \cite[Theorem 1 and Corollary 1]{ZG3}. The fact
$u\left(  x,t\right)  \in\widetilde{H}_{\infty}\left(  \mathbb{Q}_{p}%
^{n}\right)  $, for $t>0$, follows from Lemma \ref{lemma8A} (2).
\end{proof}

\begin{remark}
\label{Taibleson2}Theorem \ref{Theo6} is valid for the Taibleson operator,
i.e. if we replace the operator $f\left(  D,\alpha\right)  $ by $D_{T}%
^{\alpha}$ in\ the statement of Theorem \ref{Theo6}, then announced conclusion
is valid. The existence of a classical solution follows from the results of
\cite{R-Z0}, and the fact that $u\left(  x,t\right)  \in\widetilde{H}_{\infty
}$ follows from Lemma \ref{lemma8} and Remark \ref{note5}.
\end{remark}

\section{\label{Sect7}Pseudo-Differential Operators with Variable
Semi-Quasi-Elliptic Symbols}

In this section we consider pseudo-differential operators of the form
\[
\boldsymbol{F}(D;\alpha;x)\phi:=\mathcal{F}^{-1}(|F(\xi,x)|_{p}^{\alpha
}\mathcal{F}\left(  \phi\right)  )\text{, }%
\]
where $\alpha>0$, $\phi$ is a Lizorkin type function, and with $F(\xi,x)$ as
in (\ref{QS_symbol}). We call an extension of $\boldsymbol{F}(D;\alpha;x)$
a\textit{\ pseudo-differential operator with semi-quasielliptic symbol}. In
this section we determine the functions spaces where the equation
$\boldsymbol{F}(D;\alpha;x)u=v$ has a solution.

\subsection{Sobolev-type spaces}

\begin{definition}
\label{def_space2}Given $\beta\geq0$ and $\Xi\left(  \xi\right)  $ as before,
for $\phi\in S$, we define the following norm on $S$:
\[
||\phi||_{(\beta,\Xi)}^{2}=\int\limits_{\mathbb{Q}_{p}^{n}}[\max(1,\Xi
(\xi))]^{2\beta}|\mathcal{F}(\phi)(\xi)|^{2}d^{n}\xi.
\]

Set
\[
{\LARGE \Phi}_{M_{0}}:=\{\phi\in S;\mathcal{F}\left(  \phi\right)
|_{B_{M_{0}}}\equiv0\},
\]
where $B_{M_{0}}\left(  0\right)  =\{\xi\in\mathbb{Q}_{p}^{n};||\xi||_{p}\leq
p^{M_{0}}\}$, and the constant $M_{0}$ comes from inequality
(\ref{QS_inequality}). We define $\widetilde{H}_{(\beta,M_{0})}$ as the
completion of $\left(  {\LARGE \Phi}_{M_{0}},||\cdot||_{(\beta,\Xi)}\right)  $.
\end{definition}

\begin{remark}
If $\gamma\leq\beta$, then $||\phi||_{(\gamma,\Xi)}\leq||\phi||_{(\beta,\Xi)}$
for all $\phi\in{\LARGE \Phi}$. Also, $||\phi||_{L^{2}}\leq||\phi
||_{(\beta,\Xi)}$ for all $\beta\geq0$.
\end{remark}

\begin{lemma}
\label{lemma9}(1) If $\beta>\frac{n\max_{i}\left\{  w_{i}\right\}  }{2d}$,
then $I(\beta):=\int_{\mathbb{Q}_{p}^{n}}[\max(1,\Xi(\xi))]^{-2\beta}d^{n}%
\xi<\infty$.

\noindent(2) Take $\beta>\frac{n\max_{i}\left\{  w_{i}\right\}  }{2d}$. If
$\phi\in\widetilde{H}_{(\beta,M_{0})}$, then $\phi$\ is a uniformly continuous function.
\end{lemma}

\begin{proof}
(1) We first note that%
\begin{equation}
\Xi\left(  \xi\right)  \geq\left\Vert \xi\right\Vert _{p}^{\frac{d}{\max
_{i}\left\{  w_{i}\right\}  }}\text{, for }\left\Vert \xi\right\Vert
_{p}>1\text{.} \label{ineq6}%
\end{equation}
On the other hand,%
\begin{align*}
I(\beta)  &  =\int\limits_{\left\Vert \xi\right\Vert _{p}\leq1}[\max(1,\Xi
(\xi))]^{-2\beta}d^{n}\xi+\int\limits_{\left\Vert \xi\right\Vert _{p}>1}%
[\max(1,\Xi(\xi))]^{-2\beta}d^{n}\xi\\
&  :=I_{1}(\beta)+I_{2}(\beta).
\end{align*}

Since $[\max(1,\Xi(\xi))]^{-2\beta}$ is a continuous function, $I_{1}%
(\beta)<\infty$. We now note that $\left\Vert \xi\right\Vert _{p}>1$\ implies
$\Xi(\xi)>1$, and by (\ref{ineq6}), one gets%
\[
I_{2}(\beta)\leq\int\limits_{\left\Vert \xi\right\Vert _{p}>1}\left\Vert
\xi\right\Vert _{p}^{\frac{-2\beta d}{\max_{i}\left\{  w_{i}\right\}  }}%
d^{n}\xi\text{,}%
\]
and this last integral converges if $\frac{2\beta d}{\max_{i}\left\{
w_{i}\right\}  }>n$.

(2) It is sufficient to show that $\mathcal{F}\left(  \phi\right)  \in L^{1}$:%

\[
\int\limits_{\mathbb{Q}_{p}^{n}}\left\vert \mathcal{F}\left(  \phi\right)
\left(  \xi\right)  \right\vert d^{n}\xi\leq\left\Vert \phi\right\Vert
_{\beta}\sqrt{I(\beta)}<\infty\text{.}%
\]

\end{proof}

\begin{theorem}
\label{Theo9} Set $\beta\geq\alpha$. The operator
\[%
\begin{array}
[c]{llll}%
\boldsymbol{F}(D;\alpha;x): & \widetilde{H}_{(\beta,M_{0})} & \rightarrow &
\widetilde{H}_{(\beta-\alpha,M_{0})}\\
&  &  & \\
& \varphi & \mapsto & \boldsymbol{F}(D;\alpha;x)\varphi
\end{array}
\]
is a bicontinuous isomorphism of Banach spaces. In addition, if $\beta
>\alpha+\frac{n\max_{i}\left\{  w_{i}\right\}  }{2d}$ and $\phi\in
\widetilde{H}_{(\beta,M_{0})}$, then $\phi$\ is a uniformly continuous function.
\end{theorem}

\begin{proof}
By using the same ideas given in the proofs of Lemmas (\ref{lemma5}%
)-(\ref{lemma6}), one shows that the operator $\boldsymbol{F}(D;\alpha
;x):\widetilde{H}_{(\beta,M_{0})}\rightarrow\widetilde{H}_{(\beta-\alpha
,M_{0})}$, $\varphi\rightarrow\boldsymbol{F}(D;\alpha;x)\varphi$ is
well-defined and continuous. In order to prove the surjectivity, set $\phi
\in\widetilde{H}_{(\beta,M_{0})}$, then there exists a Cauchy sequence
$\{\phi_{l}\}\subseteq{\LARGE \Phi}$ converging to $\phi$, i.e. $\phi_{l}$
$\underrightarrow{||\cdot||_{\beta}}$ $\phi$. For each $\phi_{l}$ we define
$u_{l}$ as follows:
\[
\mathcal{F}(u_{l})(\xi)=%
\begin{cases}
\frac{\mathcal{F}\left(  \phi_{l}\right)  (\xi)}{|F(\xi,x)|_{p}^{\alpha}} &
||\xi||_{p}\geq p^{M_{0}+1}\\
0 & ||\xi||_{p}\leq p^{M_{0}}.
\end{cases}
\]

By using the argument given in the proof of Theorem \ref{surjective}, one gets
that $\left\{  u_{l}\right\}  $ is a Cauchy sequence, and if $u=\lim
_{l\rightarrow\infty}u_{l}$, then $\boldsymbol{F}(D;\alpha;x)u=\phi$. The
injectivity of $\boldsymbol{F}(D;\alpha;x)$ is established as in the proof of
Theorem \ref{surjective}. Finally, the last assertion is Lemma \ref{lemma9}.
\end{proof}

\subsection{Invariant spaces over the action of $\boldsymbol{F}(D;\alpha;x)$.}

We consider$\ \cap_{\beta\in\mathbb{N}}\widetilde{H}_{(\beta,M_{0})}$ as a
locally convex space, the topology comes from the metric (\ref{metric}). Set
$\widetilde{H}_{(\infty,M_{0})}:=\overline{(\cap_{\beta\in\mathbb{N}%
}\widetilde{H}_{(\beta,M_{0})},\rho)}$ for the completion of $(\cap_{\beta
\in\mathbb{N}}\widetilde{H}_{(\beta,M_{0})},\rho)$, and $\overline
{({\LARGE \Phi}_{M_{0}},\rho)}$ for the completion of $({\LARGE \Phi}_{M_{0}%
},\rho)$.

Propositions \ref{espacio}, \ref{Prop3}, and Theorem \ref{Theo6} have a
counterpart in the spaces $\widetilde{H}_{(\beta,M_{0})}$, $\widetilde
{H}_{(\infty,M_{0})}$, and the corresponding proofs are easy variations of the
proofs given in Section \ref{invariant}. For instance,%

\[
\widetilde{H}_{(\infty,M_{0})}=\overline{({\LARGE \Phi}_{M_{0}},\rho)}%
=\bigcap\limits_{\beta\in\mathbb{N}}\widetilde{H}_{(\beta,M_{0})}%
\]
as complete metric spaces. By Lemma \ref{lemma9} all the elements of
$\widetilde{H}_{(\infty,M_{0})}$ are uniformly continuous functions. The proof
of the following results follows from Lemma \ref{lemma9} by using the proof of
Theorem \ref{Theo6}.

\begin{theorem}
\label{Theo10} The operator
\[%
\begin{array}
[c]{llll}%
\boldsymbol{F}(D;\alpha;x): & \widetilde{H}_{(\infty,M_{0})} & \rightarrow &
\widetilde{H}_{(\infty,M_{0})}\\
&  &  & \\
& \phi & \mapsto & \boldsymbol{F}(D;\alpha;x)\phi,
\end{array}
\]
is a bicontinuous isomorphism of locally convex vector spaces.
\end{theorem}


\begin{thebibliography}{99}                                                                                               %


\bibitem {A-K-S-1}Albeverio, S., Khrennikov, A. Yu., Shelkovich, V. M., Theory
of $p$-adic distributions: linear and nonlinear models. London Mathematical
Society Lecture Note Series, 370. Cambridge University Press, Cambridge, 2010.

\bibitem {A-K-S}Albeverio, S., Khrennikov, A. Yu., Shelkovich, V. M., Harmonic
analysis in the $p$-adic Lizorkin spaces: fractional operators,
pseudo-differential equations, $p$-adic wavelets, Tauberian theorems, J.
Fourier Anal. App. \textbf{12}:4 (2006), 393-425.

\bibitem {B-S}Borievich Z. I., Shafarevich I. R., Number Theory, Academic
Press, New York, (1966).

\bibitem {D-K-K}Dragovich B., Khrennikov A. Yu., Kozyrev S. V., Volovich I.
V., On $p$-adic mathematical physics, P-Adic Numbers Ultrametric Anal. Appl. 1
(2009), no. 1, 1--17.

\bibitem {G-V}Gindikin S., Volevich L. R., The method of Newton's polyhedron
in the theory of partial differential equations. Mathematics and its
Applications (Soviet Series), 86. Kluwer Academic Publishers Group, Dordrecht, 1992.

\bibitem {H}Haran Shai, Quantizations and symbolic calculus over the $p$-adic
numbers, Ann. Inst. Fourier 43, No.4, 997-1053 (1993).

\bibitem {Koch}Kochubei A. N., Pseudo-differential equations and stochastics
over non-Archimedian fields, Pure and Applied Mathematics 244, Marcel Dekker,
New York, 2001.

\bibitem {R-Z0}Rodríguez-Vega, J. J., Zúñiga-Galindo, W. A., Taibleson
operators, $p$-adic parabolic equations and ultrametric diffusion, Pacific J.
Math. 237 (2008), no. 2, 327--347.

\bibitem {R-Z1}Rodríguez-Vega, J. J., Zúñiga-Galindo W. A., Elliptic
pseudodifferential equations and Sobolev spaces over $p$-adic fields. Pacific
J. Math., 246 (2010), no. 2, 407--420.

\bibitem {Ru}Rudin W., Functional Analysis, McGraw-Hill Company, (1973).

\bibitem {VZ}Veys W. and Zúñiga-Galindo W. A, Zeta functions for analytic
mappings, log-principalization of ideals, and newton polyhedra, Trans. Amer.
Math. Soc. 360 (2008), 2205-2227.

\bibitem {V-V-Z}Vladimirov V. S., Volovich I. V., Zelenov E. I., $p$-adic
analysis and mathematical physics, Series on Soviet and East European
Mathematics \textbf{1}, World Scientific, River Edge, NJ, 1994.

\bibitem {ZG3}Zúñiga-Galindo W.A., Parabolic Equations and Markov Processes
Over $p-$adic Fields,\ Potential Anal. 28, (2008), 185-200.

\bibitem {ZG4}Zúñiga-Galindo W.A., Local Zeta Functions Supported on Analytic
Sets and Newton Polyhedra, International Mathematics Research Notices 2009;
doi: 10.1093/imrn/rnp035.
\end{thebibliography}
\end{document}